\documentclass[onefignum,onetabnum]{siamonline171218}

\usepackage{graphicx,epstopdf}
\usepackage{amsmath,amssymb}
\usepackage{bm,paralist,enumitem}
\usepackage{placeins}
\usepackage{pgfplots}  
\usepackage{subfig}
\pgfplotsset{compat=1.9}
\usepackage{xcolor}
\usepackage{tikz}
\usetikzlibrary{plotmarks,backgrounds}
\usepgfplotslibrary{external}

\usepackage{todonotes}

\usepackage{booktabs} 
\usepackage{array} 
\usepackage{paralist} 
\usepackage{verbatim} 
\usepackage{subfig} 

\numberwithin{equation}{section}
\numberwithin{theorem}{section}

\graphicspath{{./Figures/}}

\usepackage{lipsum}
\usepackage{amsfonts}
\usepackage{graphicx}
\usepackage{epstopdf}
\usepackage{algorithmic}
\ifpdf
  \DeclareGraphicsExtensions{.eps,.pdf,.png,.jpg}
\else
  \DeclareGraphicsExtensions{.eps}
\fi

\usepackage{enumitem}
\setlist[enumerate]{leftmargin=.5in}
\setlist[itemize]{leftmargin=.5in}

\newsiamremark{remark}{Remark}
\newsiamremark{example}{Example}
\newsiamremark{assumption}{Assumption}

\headers{Asymptotic analysis of multilevel best linear unbiased estimators}{Daniel Schaden and Elisabeth Ullmann}

\title{Asymptotic Analysis of Multilevel Best Linear Unbiased Estimators\thanks{Submitted to the editors DATE.
\funding{The authors gratefully acknowledge the support by the Deutsche Forschungsgemeinschaft (DFG) through the International Research Training Group IGDK 1754 ``Optimization and Numerical Analysis for Partial Differential Equations with Nonsmooth Structures'' Projektnummer 188264188/GRK1754.}}}

\author{Daniel Schaden\thanks{Department of Mathematics, Technical University of Munich, Boltzmannstr. 3, 85748 Garching b. M\"unchen, Germany, (\email{schaden@ma.tum.de},\email{eullmann@ma.tum.de}).}
\and Elisabeth Ullmann\footnotemark[2]}

\usepackage{amsopn}

\newcommand{\mce}{ \widehat{\mu}^{\operatorname{MC}}} 
\newcommand{\blue}{ \widehat{\mu}^{\operatorname{B}}}

\newcommand{\mlmce}{ \widehat{\mu}^{\operatorname{MLMC}}}
\newcommand{\mfmce}{ \widehat{\mu}^{\operatorname{MFMC}}}

\newcommand{\acvmfe}{\widehat{\mu}^{\operatorname{ACV-MF}}}

\newcommand{\ree}[1]{ \widehat{\mu}^{\operatorname{RE, #1}}}
\newcommand{\saob}{ \widehat{\mu}^{\operatorname{SAOB}}}
\newcommand{\saobk}[1]{ \widehat{\mu}^{\operatorname{SAOB, #1}}}

\ifpdf
\hypersetup{
  pdftitle={Asymptotic analysis of multilevel best linear unbiased estimators},
  pdfauthor={Daniel Schaden and Elisabeth Ullmann}
}
\fi

\begin{document}

\maketitle

\begin{abstract}
We study the computational complexity and variance of multilevel best linear unbiased estimators introduced in [D. Schaden and E. Ullmann, \textit{SIAM/ASA J. Uncert. Quantif.}, (2020)]. 
We specialize the results in this work to PDE-based models that are parameterized by a discretization quantity, e.g., the finite element mesh size. 
In particular, we investigate the asymptotic complexity of the so-called sample allocation optimal best linear unbiased estimators (SAOBs).
These estimators have the smallest variance given a fixed computational budget. 
However, SAOBs are defined implicitly by solving an optimization problem and are difficult to analyze. 
Alternatively, we study a class of auxiliary estimators based on the Richardson extrapolation of the parametric model family. 
This allows us to provide an upper bound for the complexity of the SAOBs, showing that their complexity is optimal within a certain class of linear unbiased estimators.
Moreover, the complexity of the SAOBs is not larger than the complexity of Multilevel Monte Carlo.
The theoretical results are illustrated by numerical experiments with an elliptic PDE. 
\end{abstract}



\section{Introduction}\label{sec:intro}
A common model in the field of uncertainty quantification (UQ) is a partial differential equation (PDE) equipped with random coefficients, and other inputs whose uncertainty is modeled by a probability distribution on a suitable function space.
An important building block in UQ is the estimation of expected values of output quantities of interest linked with such random PDEs.
Monte Carlo (MC) estimators are often infeasible in this situation due to the high cost per sample.
In the last decade \textit{multilevel} estimators have been developed to address this problem and provide estimates with a much smaller computational cost.
This is achieved by \textit{variance reduction} and by working with a collection of PDE models with different resolutions or fidelities.
Typically, multilevel estimators couple an expensive, high resolution PDE model with cheap, low resolution PDE models.
Examples for multilevel estimators are multilevel Monte Carlo (MLMC) \cite{giles_2008, giles_2015}, multifidelity Monte Carlo (MFMC) \cite{Peherstorfer_2016, Peherstorfer_2018b} and approximate control variates (ACVs) \cite{Gorodetsky_2020}.
In this work we revisit the multilevel best linear unbiased estimator (BLUE) introduced in \cite{SchadenUllmann:2020}.
The multilevel BLUE is a linear, unbiased combination of MC estimators associated with different PDE resolutions.
Importantly, the multilevel BLUE selects the \textit{optimal} linear combination of samples with respect to the estimator variance.

The analysis in \cite{SchadenUllmann:2020} is independent of the underlying models used for the multilevel BLUE.
Now we specialize the results in \cite{SchadenUllmann:2020}, and assume that we work with PDE-based models with random coefficients.
For discretized PDEs it is often the case that the fidelity of a model can be linked with a discretization parameter, e.g., the mesh size of a finite element space.
Thus, it is natural to assume that the models are parameterized by the mesh size, where a small parameter value gives accurate results which might however be expensive to compute.
We make the form of the parameterization precise in the following sections of the paper.

In \cite{SchadenUllmann:2020} we introduced a sample allocation optimal estimator termed SAOB which is a special case of the multilevel BLUE.
The SAOB selects so called model groups and the number of samples for each group such that the estimator variance is minimal given a fixed computational budget.
The SAOB is implicitly defined in terms of the solution of an optimization problem.
This complicates its analysis, even if we make more assumptions on the models we work with.
As an auxiliary tool we now introduce and study a class of estimators termed RE estimators.
These are based on Richardson extrapolation (RE) \cite{Brezinski_2000, Richardson_1911}, and on the telescoping sum approach in MLMC \cite{giles_2008, giles_2015}.
It turns out that the classical MLMC estimator in \cite{giles_2008, giles_2015} is a special case of an RE estimator.

RE is a well known technique in numerical analysis.
It employs linear combinations of a family of approximations to improve the accuracy of the individual approximations within the family.
Perhaps the most widely known application of RE is the Romberg method \cite{Romberg_1955} for numerical quadrature.
In addition, RE has been used for ordinary differential equations \cite{Bulirsch_1966}, stochastic differential equations (SDEs) \cite{TalayTubaro1990}, and partial differential equations (see e.g. \cite{Asadzadeh_2009, Blum_1986, Rannacher_1988}).
By combining RE with the MLMC complexity theory \cite{Cliffe_2011, giles_2015} we provide an upper bound on the computational complexity of the SAOB and show that its asymptotic complexity is not larger than the complexity of MLMC.
We further show that the coefficients of the SAOB converge to the coefficients in the linear combination of the RE estimator for an academic toy problem.
Hence we can expect close links between the RE estimator and the SAOB.

	The combination of RE and multilevel estimators has been discussed in a few places in the literature.
	However, these are typically restricted to bias errors of a specific form and SDE discretizations.
	Multilevel Richardson--Romberg extrapolation has been explored already by Giles in the pioneering MLMC paper \cite{giles_2008} where one level of RE was used.
	In \cite{Mueller2015} the authors combine RE and MLMC for specific discretizations of the Langevin equation.
    Lemaire and Pag\`{e}s \cite{lemaire_2017} introduce a multilevel Richardson--Romberg estimator termed ML2R for SDE discretizations where the bias error w.r.t. the discretization parameter $h$ has the form $h^{\gamma^k}$ with linearly growing exponents $\gamma^k = k \alpha$.
    An antithetic extension of the ML2R estimator is studied in \cite{Mbaye2017}.
	In our work we study RE estimators to analyze the complexity of the SAOB.
	In fact, this idea grew out of numerical experiments where we observed that the coefficients of the multilevel BLUE approached RE coefficients.
	This highlights a novel application of RE in the study of multilevel estimators and uncertainty quantification.

The main contributions of this paper are as follows: $(i)$ a general upper bound on the complexity of the SAOB, $(ii)$ a specific upper bound in terms of RE estimators for parametric model families and $(iii)$ a complete complexity and variance analysis of the RE estimators.

The remainder of this work is structured as follows.
In \Cref{section_problem_formulation} we review the necessary definitions and results introduced in \cite{SchadenUllmann:2020}, in particular, the multilevel BLUE and the SAOB. 
We further present an asymptotic cost bound for the SAOB, however, this bound is difficult to verify in practice.
In \Cref{section_upper_cost_bound} we provide upper bounds for the cost of the SAOB which are easier to verify.
These bounds are based on MC, MLMC and RE estimators. 
In \Cref{sec:numerics} we verify the theoretical results by numerical experiments for a PDE-based quantity of interest in two space dimensions.
In \Cref{section_re_convergence_to_saob} we show that the SAOBs converge to RE estimators for an academic toy problem. 
In \Cref{section_acv} we compare the SAOBs and RE estimators with the ACV estimators. \Cref{sec:conclusions} provides concluding remarks.
\section{Problem formulation}\label{section_problem_formulation}

%
Let $Z$ be a scalar-valued random variable whose expectation $\mathbb{E}[Z]$ we want to estimate. 
We assume that exact sampling from $Z$ is not possible.
Hence we work with a family of approximations $Z_1, \dots , Z_L$, $L\in \mathbb{N}$. 
Typical scenarios we have in mind are finite element based PDE discretizations where $\ell$ denotes the level of mesh refinement and $Z_\ell(\omega)$ is an output quantity of interest which requires solving the discretized PDE with random inputs depending on the event $\omega$.
In this case, for a large discretization parameter $\ell$ the approximation of $Z(\omega)$ by $Z_\ell(\omega)$ is accurate yet computationally expensive. 

Let us define the expectation and covariance for the random variables $Z_1,\dots, Z_L$ as follows,
\begin{align*}
\mu_\ell := \mathbb{E}[Z_\ell], \quad c_{\ell, j} := \operatorname{Cov}(Z_\ell, Z_j), \quad \text{for all } \ell, j \in \{1, \dots , L\}.
\end{align*} 
We further define the \textit{mean vector} $\mu := (\mu_1, \dots, \mu_L)^T $ and the \textit{model covariance matrix} $C := (c_{\ell, j})_{\ell, j = 1}^L$. 
We assume throughout this paper that the expectations and variances of $Z, Z_1, \dots, Z_L$ exist and are finite. 
For every subset $Q \subseteq \{1, \dots, L\}$ we define the principal submatrix of $C$ as $$C_{Q, Q} := (c_{\ell, j})_{\ell, j \in Q} \in \mathbb{R}^{|Q| \times |Q|}. $$ 
We use a similar notation for vectors $\alpha \in \mathbb{R}^L$  with $\alpha_Q := (\alpha_\ell)_{\ell \in Q} \in \mathbb{R}^{|Q|}$. 
We further introduce the notation $\widehat{\mu}_\alpha$ for unbiased estimators of $\alpha^T \mu$, i.e., estimators which satisfy $\mathbb{E}[\widehat{\mu}_\alpha] = \alpha^T \mu$. 
In the special case $\alpha = e_\ell$, that is, $\alpha$ is the $\ell$th unit vector, we write  $\widehat{\mu}_\ell$ for an unbiased estimator of $e_\ell^T \mu = \mu_\ell$. 
Finally, we use the generic constant $c$ for estimates, i.e. if it holds
\begin{equation*}
\phi(\ell) \leq c 2^{- 2 \ell},
\end{equation*}
for some function $\phi$, then $c$ is independent of $\ell$. 
Moreover, the value of the generic constant $c$ may change from equation to equation. 
\subsection{Multilevel best linear unbiased estimator}\label{section_blue}
Our goal is to construct a variance minimal, linear, unbiased estimator for $\alpha^T \mu$, where the vector $\alpha \in \mathbb{R}^L$ is given. 
Recall that the estimator $\widehat{\mu}_\alpha$ is \textit{linear and unbiased} if there exist coefficients $\beta_1, \dots, \beta_N \in \mathbb{R}$ and samples $Z_{\ell_1}(\omega_{j_1}), \dots, Z_{\ell_N}(\omega_{j_N})$ for $\ell_1, \dots, \ell_N \in \{1, \dots, L\}$ such that
\begin{equation*}
\widehat{\mu}_\alpha = \sum_{i = 1}^N \beta_i Z_{\ell_i}(\omega_{j_i}), \quad \mathbb{E}[\widehat{\mu}_\alpha] = \alpha^T \mu.
\end{equation*}
The events $\omega_{j_i}$ determine the correlation structure between the samples and thus the variance of the estimator. 
In particular, not all events are necessarily distinct and the same event might be used with different output quantities.
Throughout this paper we use the framework established in \cite{SchadenUllmann:2020}. 
Let $K := 2^L - 1$, and let $S^1, \dots, S^K \in 2^{\{1,\dots, L\}} \setminus \{\emptyset\}$ denote the non-empty subsets of the index set $\{1,\dots, L\}$.
In particular,
\begin{equation*}
S^j \not = S^k \quad \text{ for } j \not = k.
\end{equation*}   
Each collection of indices $S^k$ is a so-called \textit{model group}. 
For an event $\omega$ we define the corresponding model group evaluation, the covariance matrix, the restriction and prolongation matrices, respectively, as follows,
\begin{align*}
Z^k(\omega) &:= (Z_\ell(\omega))_{\ell \in S^k} \in \mathbb{R}^{|S^k|}, \quad &C^k &:= C_{S^k, S^k} = \operatorname{Cov}(Z^k, Z^k) \in \mathbb{R}^{|S^k| \times |S^k|}, \\
R^k v &:= v_{S^k} \in \mathbb{R}^{|S^k|} \text{ for all } v \in \mathbb{R}^L, \quad &P^k &:= (R^k)^T \in \mathbb{R}^{L \times |S^k|}, \quad k=1, \dots, K.
\end{align*}  
In the remainder of this paper we assume that each matrix $C^k$ defined above is regular.
Next we consider $m_k \in \mathbb{N}_0$ independent samples associated with each model group $S^k$, denoted by  $\omega_i^k$, $i=1,\dots,m_k$.
We further assume that $\omega_i^k$ and $\omega_i^j$ are independent for $k \not = j$. 
We construct a linear unbiased estimator using the samples
\begin{equation} \label{equation_available_samples}
(Z^1(\omega^1_i))_{i = 1}^{m_1}, \dots, (Z^K(\omega^K_i))_{i = 1}^{m_K}.
\end{equation}
Note that the model group $S^k$ describes the statistical coupling of models $Z_\ell$ with $\ell \in S^k$ and that $m_k$ is the number of independent evaluations of $Z^k$. 
We organize the samples in \eqref{equation_available_samples} in a block vector and define a block linear model as follows,
\begin{equation} \label{equation_linear_regression}
\begin{aligned}
Y &= H \mu + \eta, \\
Y &= (Y^k)_{k = 1}^K, \quad &H &= (H^k)_{k = 1}^K, \quad &\eta &= (\eta^k)_{k = 1}^K, \\
Y^k &= (Z^k(\omega^k_i))_{i = 1}^{m_k}, \quad &H^k &= (R^k)_{i = 1}^{m_k}, \quad &\eta^k &= (Z^k(\omega^k_i) - R^k \mu)_{i = 1}^{m_k},
\end{aligned}
\end{equation} 
where $Y$ is the vector of observations, $H$ is the design matrix describing a linear relationship between the unknown $\mu$ and observations, and $\eta$ is a noise vector with mean zero. 
This is a generalized linear model, where the covariance matrix of the noise $\eta$ is block diagonal due to the assumed independence structure of the samples,
\begin{align*}
\operatorname{Cov}(\eta, \eta) &= \operatorname{diag}((\operatorname{Cov}(\eta^k, \eta^k))_{k = 1}^K) = \operatorname{diag}(\operatorname{diag}((\operatorname{Cov}(Z^k(\omega^k_i), Z^k(\omega^k_i)))_{i = 1}^{m_k})_{k = 1}^K) \\
&= \operatorname{diag}(((C^k)_{i = 1}^{m_k})_{k = 1}^K).
\end{align*} 
Finally, let $m = (m_1, \dots, m_K)^T$ denote the vector collecting the number of samples in each model group, and define the matrix and vector
\begin{equation*}
\Psi(m) := \sum_{k = 1}^K m_k P^k (C^k)^{-1} R^k \in \mathbb{R}^{L \times L}, \quad 
y(m) := \sum_{k = 1}^K P^k (C^k)^{-1} \sum_{i = 1}^{m_k} Z^k(\omega_i^k) \in \mathbb{R}^L.
\end{equation*} 
Then, the best linear unbiased estimator (BLUE) $\blue(m)$ for the mean vector $\mu$ in \eqref{equation_linear_regression} associated with the number of samples $m$ is the solution of the normal equations
\begin{equation}\label{blue}
\Psi(m) \blue(m) = y(m).
\end{equation}
Under certain assumptions the scalar value $\blue_\alpha(m) := \alpha^T \blue(m)$ is also the BLUE for $\alpha^T \mu$ and the estimator variance is $\operatorname{Var}(\blue_\alpha(m)) = \alpha^T \Psi(m)^{-1} \alpha$ (see \cite[Theorem 2.7]{SchadenUllmann:2020}).

\subsection{Sample allocation optimal BLUE}\label{section_saob}
Observe that the estimator $\blue_\alpha(m)$ depends on the number of samples $m_k$ for each model group $S^k$.
Now we want to select $m$ optimally given a fixed computational budget. 
We assume costs $w_\ell > 0$ to compute a sample of $Z_\ell$, $\ell=1,\dots,L$.
This incurs the cost  
\begin{equation*}
W_k := \sum_{\ell \in S^k} \operatorname{Cost}(Z_\ell) = \sum_{\ell \in S^k} w_\ell
\end{equation*}  
to evaluate all models in the group $S^k$, $k=1,\dots,K$.
For a fixed computational budget $p > 0$ we then solve the following \textit{sample allocation problem}:
\begin{equation}\label{min}
\begin{cases}
\quad \min_{m \in \mathbb{N}_0^K} \quad \operatorname{Var}(\blue_\alpha(m)) &= \alpha^T \Psi(m)^{-1} \alpha \\
\quad \text{s.t.} \quad  \operatorname{Cost}(\blue_\alpha(m)) &= \sum_{k = 1}^K m_k W_k \leq p, \\
\quad m_k &= 0, \quad \text{if } |S^k| > q.
\end{cases}
\end{equation}
In \eqref{min}, the coupling number $q > 0$ defines the maximal number of models that are evaluated for the same input $\omega$. 
We then define the estimator termed $\operatorname{SAOB}, q$ as
\begin{equation}\label{saob}
\saobk{q}_\alpha = \alpha^T \blue(m^*),
\end{equation}
where $m^*$ is a minimizer of \eqref{min}.
If the coupling number $q = +\infty$, we drop $q$ in the notation above.
It can be proved that the $\operatorname{SAOB}$ is variance minimal in the class of linear unbiased estimators with costs bounded by $p$.
\begin{theorem}[{\cite[Theorem 3.2]{SchadenUllmann:2020}}] \label{theorem_saob_variance_minial}
	Let the model covariance matrix $C$ be positive definite. 
	Then any linear unbiased estimator $\widehat{\mu}_\alpha$ that uses the samples in \eqref{equation_available_samples} for any $m$ such that $\operatorname{Cost}(\widehat{\mu}_\alpha) \leq p$ satisfies 
	\begin{equation*}
	\operatorname{Var}(\widehat{\mu}_\alpha) \geq \operatorname{Var}(\saob_\alpha).
	\end{equation*}
\end{theorem}
Note that the $\operatorname{SAOB}$ is difficult to analyze since it depends on a minimizer $m^*$ of the problem \eqref{min}, and is constructed \textit{implicitly} with the solution of a linear regression problem.
To make progress we now present some general results on the complexity of linear unbiased estimators, keeping in mind that the $\operatorname{SAOB}$ is a special case.
\subsection{Asymptotic analysis of linear unbiased estimators}
We are interested in linear unbiased estimators of $\alpha^T \mu$ that form a linear combination of correlated MC estimators
\begin{equation} \label{equation_estimator_general}
\widehat{\mu}_\alpha = \sum_{k = 1}^K \sum_{\ell \in S^k} \beta_\ell^k \frac{1}{m_k} \sum_{i=1}^{m_k}Z_\ell(\omega_i^k).
\end{equation}
If the model group $S^k$ is not used and thus $m_k = 0$, we define $\beta^k = 0$ and $\beta_\ell^k / m_k = 0$. 
This is equivalent to excluding the $k$th summand in \eqref{equation_estimator_general}. 
We further define $\beta^k_\ell = 0$ for $\ell \not \in S^k$. 
Combining the definitions in \cite[Equ. (2.11)]{SchadenUllmann:2020}, \cite[Equ. (2.7)]{SchadenUllmann:2020}, and rearranging, we see that the coefficients $\beta^k$ in \eqref{equation_estimator_general} associated with $\widehat\mu_\alpha^{\text{SAOB}}$ can be written as
\begin{equation} \label{equation_saob_coefficients}
\beta^k = m_k P^k (C^k)^{-1} R^k \Psi(m)^{-1} \alpha
\end{equation}
provided that $\Psi(m)$ is regular.
Recall that $Z_\ell(\omega_i^k)$ and $Z_\ell(\omega_j^n)$ are statistically independent if $i\neq j$ or $k \neq n$. 
Thus the estimator variance is equal to
\begin{equation} \label{equation_variance_general}
\operatorname{Var}(\widehat{\mu}_\alpha) = \sum_{k=1}^K \frac{1}{m_k} \operatorname{Var}\left(\sum_{\ell \in S^k} \beta^k_\ell Z_\ell\right) = \sum_{k=1}^K \frac{1}{m_k} (\beta^k)^T C \beta^k.
\end{equation}
Now, if $(\beta^k)^T C \beta^k$ is asymptotically small, then only a small number of samples $m_k$ is required to achieve a small variance contribution. Therefore we want to obtain an expression for the optimal sample allocation $m$ by minimizing \eqref{equation_variance_general} given a budget constraint
\begin{equation} \label{equation_opt_sample_allocation}
\begin{cases}
\quad \min_{m \in \mathbb{R}^K} J(m) &= \sum_{k=1}^K \frac{1}{m_k} (\beta^k)^T C \beta^k \\
\quad \text{s.t.} \sum_{k = 1}^K m_k W_k &\leq p, \\
\quad m_k &\geq 0, \quad \text{for } k = 1, \dots, K.
\end{cases}
\end{equation}
We now follow the classical MLMC approach \cite[Section 1.3]{giles_2015} to compute the unique minimizer of \eqref{equation_opt_sample_allocation}. 
Note that we replace $\operatorname{Var}(Z_\ell - Z_{\ell - 1})$ by the more general term $(\beta^k)^T C \beta^k$. 
\begin{lemma}[Optimal sample allocation] \label{lemma_opt_sample_allocation}
Let $C$ be positive definite, let $\alpha \not = 0$ and the coefficients $\beta^k_\ell$ such that the bias constraint is satisfied, i.e. $\alpha = \sum_{k = 1}^K \beta^k$. Then there exists a unique minimizer $m^*$ of \eqref{equation_opt_sample_allocation} of the form
\begin{equation} \label{equation_saob_opt_samples}
m^*_k = \frac{p}{\sum_{k = 1}^K (\left[(\beta^k)^T C \beta^k \right] W_k )^{1 / 2}} \left(\left[(\beta^k)^T C \beta^k \right] / W_k \right)^{1 / 2}
\end{equation}
with associated variance
\begin{equation} \label{equation_saob_opt_variance}
J(m^*) = \frac{1}{p} \left(\sum_{k = 1}^K \left(\left[(\beta^k)^T C \beta^k \right] W_k \right)^{1 / 2} \right)^2.
\end{equation}
In particular, the cost $p$ to achieve the variance $J(m^*) = \varepsilon^2$ is equal to
\begin{equation*}
p = \varepsilon^{-2} \left(\sum_{k = 1}^K \left(\left[(\beta^k)^T C \beta^k \right] W_k \right)^{1 / 2} \right)^2.
\end{equation*}
\end{lemma}
\begin{proof}
The proof is provided in \Cref{appendix_proof_lemma_opt_sample_allocation}. 
\end{proof}
Note that we assume real-valued numbers in \eqref{equation_saob_opt_samples} which is not practical.
However, rounding  $m^*_k$ in \eqref{equation_saob_opt_samples} for all $K = 2^L - 1$ possible values of $k$ may increase the total cost significantly, since $2^{L - 1}$ model groups contain the finest model $Z_L$. 
Fortunately, the result \cite[Theorem 3.6]{SchadenUllmann:2020} tells us that we may choose $\beta^k \not = 0$ for at most $L$ different indices $k \in \{1, \dots, K \}$ without increasing the variance. 
Without loss of generality this allows us to set $\beta^{L + 1} = \beta^{L + 2} = \dots = \beta^K = 0$ by suitably renumbering $S^1, \dots, S^K$. 
Since the estimator $\widehat{\mu}_\alpha$ is an unbiased estimator for $\alpha^T \mu$, the $\beta^k$ then satisfy
\begin{equation} \label{equation_cost_saob_conditions}
\alpha = \sum_{k = 1}^L \beta^k, \quad \beta^k_\ell = 0, \quad \ell \not \in S^k.
\end{equation}
Let us now comment on the bias which is determined by the vector $\alpha$. 
Observe that in practise it is not advisable to fix $\alpha$ independently of the finest level $L$. 
For example, for $L = 1$ we may choose $\alpha = 1$, however, for $L = 2$ the choice $\alpha = (0, 1)^T$ is suitable provided that $Z_2$ is a better approximation of $Z$ compared to $Z_1$. 
Hence we introduce a sequence of bias vectors $(\alpha^L)_{L = 1}^\infty$ that in turn defines the sequence of estimators $(\widehat{\mu}_{\alpha^L})_{L = 1}^\infty$. 
The coefficients $\beta^k$ and model groups $S^k$ clearly depend on $L$, however for ease of notation we drop this dependence. 
Furthermore, for $(\alpha^L)^T \mu$ we assume that $\alpha^L \in \mathbb{R}^L$ and $\mu = (\mathbb{E}[Z_1], \dots, \mathbb{E}[Z_L])^T \in \mathbb{R}^L$ where we again drop the dependence of $\mu$ on $L$. 
We now estimate the asymptotic complexity of the sequence of estimators $(\widehat{\mu}_{\alpha^L})_{L = 1}^\infty$.
\begin{theorem}[Asymptotic cost of linear unbiased estimators] \label{theorem_saob_complexity_rates}
	Assume that there exist positive constants $\gamma^{\operatorname{Bias}}$, $\gamma^{\operatorname{Var}}$ and $\gamma^{\operatorname{Cost}}$, such that the following statements hold for all $L \in \mathbb{N}$:
	\begin{align}
	\tag{M0} \label{equation_M0_saob} \alpha^L &= \sum_{k = 1}^L \beta^k, \quad \quad \beta^k_j = 0 \text{ if } j \not \in S^k &&\quad  \text{ for all } k \leq L, \\
	\tag{M1} \label{equation_M1_saob} |(\alpha^L)^T \mu - \mathbb{E}[Z]| &\leq c 2^{- L \gamma^{\operatorname{Bias}}}, \\
	\tag{M2} \label{equation_M2_saob} (\beta^k)^T C \beta^k &\leq c 2^{- k \gamma^{\operatorname{Var}}} &&\quad \text{ for all }  k \leq L, \\
	\tag{M3} \label{equation_M3_saob} W_k &\leq c 2^{k \gamma^{\operatorname{Cost}}} &&\quad \text{ for all }  k \leq L, 
	\end{align}		
	where the constant $c > 0$ is independent of $L$. 
	Then there exists a level $L \in \mathbb{N}$, model groups $S^1, \dots, S^L$ and numbers of samples $m_1, \dots, m_L$ to achieve $\mathbb{E}[|\widehat{\mu}_{\alpha^L} - \mathbb{E}[Z]|^2] \leq \varepsilon^2$ with a cost bounded by
	\begin{equation} \label{equation_saob_complexity_expression}
	\operatorname{Cost}(\widehat{\mu}_{\alpha^L}) \leq c \varepsilon^{- \gamma^{\operatorname{Cost}} / \gamma^{\operatorname{Bias}}} + c \begin{cases}
	\varepsilon^{-2}, \quad &\text{if } \gamma^{\operatorname{Cost}} < \gamma^{\operatorname{Var}}, \\
	\varepsilon^{-2} (\log \varepsilon)^2, \quad &\text{if } \gamma^{\operatorname{Cost}} = \gamma^{\operatorname{Var}}, \\
	\varepsilon^{-2 - \frac{\gamma^{\operatorname{Cost}} - \gamma^{\operatorname{Var}}}{\gamma^{\operatorname{Bias}}}}, \quad &\text{if } \gamma^{\operatorname{Cost}} > \gamma^{\operatorname{Var}}.
	\end{cases} 
	\end{equation}
\end{theorem}
\begin{proof}
	The proof is analogous to the proof of \cite[Theorem 1]{Cliffe_2011}. 
	Note that we do not use the assumption $\gamma^{\operatorname{Bias}} \geq \min \{\gamma^{\operatorname{Var}}, \gamma^{\operatorname{Cost}}\} / 2$.
	Thus we do not explicitly bound the term $\varepsilon^{- \gamma^{\operatorname{Cost}} / \gamma^{\operatorname{Bias}}}$ accounting for ceiling the number of samples.
\end{proof}
Recall that the coefficients $\beta^k$ in \eqref{equation_saob_coefficients} and the $L$ model groups $S^k$ in \eqref{equation_estimator_general} of the $\operatorname{SAOB}$ are chosen to minimize the variance and are only given implicitly. Therefore, we typically cannot verify the assumptions \eqref{equation_M2_saob} and \eqref{equation_M3_saob} for the $\operatorname{SAOB}$. Furthermore, it is not clear how to choose the sequence of bias vectors $(\alpha^L)_{L = 1}^\infty$ such that \eqref{equation_M1_saob} holds. 
Hence \Cref{theorem_saob_complexity_rates} is of limited practical use to determine the asymptotic cost of the $\operatorname{SAOB}$.
%
\section{Upper asymptotic cost bounds} \label{section_upper_cost_bound}
We now present an alternative, constructive approach for the complexity analysis of the $\operatorname{SAOB}$. 
A straightforward consequence of \Cref{theorem_saob_variance_minial} is that the SAOB has the optimal complexity in the class of linear unbiased estimators. 
\begin{theorem}[Complexity bound for $\operatorname{SAOB}$] \label{corollary_saob_complexity_optimal}
	Let the model covariance matrix $C$ be positive definite and let $\widehat{\mu}_\alpha$ be a linear unbiased estimator that uses the samples in \eqref{equation_available_samples}. If $\widehat{\mu}_\alpha$ estimates $\mathbb{E}[Z]$ with a mean square error (MSE) bounded by $\varepsilon^2$ and with costs bounded by $\phi(\varepsilon)$, that is, 
	\begin{align*}
	\mathbb{E}[|\widehat{\mu}_\alpha - \mathbb{E}[Z]|^2] \leq \varepsilon^2 \quad \text{with} \quad
	\operatorname{Cost}(\widehat{\mu}_\alpha) \leq \phi(\varepsilon),
	\end{align*}
	then the estimator $\saob_\alpha$ achieves the same MSE with the same or smaller costs
	\begin{align*}
	\mathbb{E}[|\saob_\alpha - \mathbb{E}[Z]|^2] 
	\leq 
	\varepsilon^2 \quad \text{with} \quad
	\operatorname{Cost}(\saob_\alpha) \leq \operatorname{Cost}(\widehat{\mu}_\alpha) \leq \phi(\varepsilon).
	\end{align*}
	In particular, we may choose $p = \operatorname{Cost}(\widehat{\mu}_\alpha)$ in \eqref{min}.
\end{theorem}
\begin{proof}
	We use a bias variance decomposition and the bound on the MSE to show
	\begin{equation*}
	\mathbb{E}[(\widehat{\mu}_\alpha - \mathbb{E}[Z])^2] 
	= 
	(\alpha^T \mu - \mathbb{E}[Z])^2 + \operatorname{Var}(\widehat{\mu}_\alpha) 
	\leq 
	\varepsilon^2.
	\end{equation*}
	Now, the estimator $\saob_\alpha$ has the same bias as $\widehat{\mu}_\alpha$.
	Therefore we only compare the variance of $\saob_\alpha$ and $\widehat{\mu}_\alpha$, respectively, as unbiased estimators of $\alpha^T \mu$. 
	If we choose $p = \operatorname{Cost}(\widehat{\mu}_\alpha)$ in \Cref{theorem_saob_variance_minial}, then -- by construction -- the estimator $\saob_\alpha$ has equal or smaller variance than $\widehat{\mu}_\alpha$. 
\end{proof}
\begin{remark} \label{remark_saob_kappa_complexity}
Similarly to the proof of \cite[Theorem 3.2]{SchadenUllmann:2020} it can be shown that any linear unbiased estimator $\widehat{\mu}_\alpha$ that uses the samples in \eqref{equation_available_samples} \textit{and} couples at most $q$ models using the same event $\omega$, that is $m_k = 0$ if $|S^k| > q$, satisfies
\begin{equation*}
\operatorname{Var}(\widehat{\mu}_\alpha) \geq \operatorname{Var}(\saobk{q}_\alpha).
\end{equation*}
Then we again choose $p = \operatorname{Cost}(\widehat{\mu}_\alpha)$ in \eqref{min} and are able to conclude that $\operatorname{SAOB}, q$ is complexity optimal in the class of linear unbiased estimators that couple at most $q$ models.
\end{remark}
The complexity bound in \Cref{corollary_saob_complexity_optimal} is formulated as a \textit{comparison principle}, and does not give an explicit expression for $\phi(\varepsilon)$. 
However, explicit expressions may be obtained by the following procedure: 
\begin{itemize}
\item[1.] Let $(\widehat{\mu}_{\alpha^L})_{L = 1}^{\infty}$ be a sequence of linear unbiased estimators.
\item[2.] Bound the asymptotic complexity of $(\widehat{\mu}_{\alpha^L})_{L = 1}^{\infty}$ in terms of an explicit expression $\phi(\varepsilon)$. 
\item[3.] Verify the assumptions of \Cref{corollary_saob_complexity_optimal} for all estimators $(\widehat{\mu}_{\alpha^L})_{L=1}^\infty$. 
Then the complexity of the $\operatorname{SAOB}$ is bounded by $\phi(\varepsilon)$.
\end{itemize}
%
Of course, this approach gives only an upper cost bound for the $\operatorname{SAOB}$  that is not necessarily sharp, especially if the estimators $(\widehat{\mu}_{\alpha^L})_{L = 1}^{\infty}$ are chosen poorly. 
In the following we describe a construction that we found useful.
First, we choose the vector $\alpha^L$ such that the bias is small. 
This can often be achieved by Richardson extrapolation. 
Next, since the estimator $\widehat{\mu}_{\alpha^L}$ has to be linear and unbiased, we use \Cref{lemma_opt_sample_allocation} as a guideline. 
Observe that for a fixed vector $\alpha^L$ the coefficients of the $\operatorname{SAOB}$ satisfy
\begin{equation}
\begin{cases}
\quad \min_{\beta^1,\dots, \beta^K \in \mathbb{R}^L} J(\beta) &= \sum_{k = 1}^K \left(\left[(\beta^k)^T C \beta^k \right] W_k \right)^{1 / 2} \\
\quad \text{s.t.} \quad &\alpha^L = \sum_{k = 1}^K \beta^k, \\
\quad &\beta^k_\ell = 0 \quad \text{if } \ell \not \in S^k,
\end{cases}
\end{equation}
where $\varepsilon^{-2} J^2$ is the cost of $\operatorname{SAOB}$ to achieve a variance of $\varepsilon^2$.
Note that this does not include the cost for ceiling the number of samples. 
We may then choose suboptimal coefficients $\beta^k$ which increases the variance but gives an asymptotic expression or upper bound for $J^2$. 
The goal of this section is to demonstrate this proof strategy for different choices of $\beta^k$.
\begin{itemize}
\item The MC estimator has typically a high asymptotic cost and thus gives an upper bound that is not sharp. To demonstrate the general idea we write down the MC complexity in \Cref{subsection_monte_carlo}. 
\item An improvement is possible by making $(\beta^k)^T C \beta^k$ small such that its contribution to $J$ is small while ensuring that the bias constraint is still satisfied. 
The MLMC estimator achieves this by using a telescoping sum and coefficients $\beta^k$ such that $(\beta^k)^T C \beta^k = \operatorname{Var}(Z_\ell - Z_{\ell - 1})$. The variance of the difference is often asymptotically small with a known rate. We summarize this result in \Cref{subsection_multilevel_monte_carlo}.
\item We further improve the asymptotic complexity of the MLMC estimator by using RE to obtain even smaller asymptotic expressions for $(\beta^k)^T C \beta^k$. We formally define the resulting RE estimators in \Cref{subsection_re_estimators}. 
\end{itemize}
\subsection{Monte Carlo estimator} \label{subsection_monte_carlo}
The standard MC estimator uses $\alpha^L = e_L$, $S^1 = \{L\}$, $\beta^1 = e_L$ and $\beta^k = 0$ for all other $k$. The estimator reads
\begin{equation*}
\mce_L = \frac{1}{m_1} \sum_{i = 1}^{m_1} Z_L(\omega^1_i).
\end{equation*}
We have the following well known result for the complexity of the MC estimator. 
\begin{corollary}[MC asymptotic cost, {\cite[Subsection 2.1]{Cliffe_2011}}] \label{corollary_mc_complexity_rate}
	Assume that there exist positive constants $\gamma^{\operatorname{Bias}}$ and $\gamma^{\operatorname{Cost}}$ such that the following statements hold for all $\ell \in \mathbb{N}$,
	\begin{align}
	\tag{MC, M1} \label{equation_M1_mc} |\mathbb{E}[Z_\ell] - \mathbb{E}[Z]| &\leq c 2^{- \ell \gamma^{\operatorname{Bias}}}, \\
	\tag{MC, M2} \label{equation_M2_mc} \operatorname{Var}(Z_\ell) &\leq c, \\
	\tag{MC, M3} \label{equation_M3_mc} \operatorname{Cost}(Z_\ell) &\leq c 2^{\ell \gamma^{\operatorname{Cost}}}, 
	\end{align}		
	where the constant $c > 0$ is independent of $\ell$. 
	Then there exists a final level $L > 0$ and number of samples $m_1$ to achieve a MSE of $\mathbb{E}[|\mce_L - \mathbb{E}[Z]|^2] \leq \varepsilon^2$ with costs bounded by
	\begin{equation} \label{equation_mc_asymptotic_cost}
	\operatorname{Cost}(\mce_L) \leq c \varepsilon^{-2 - \gamma^{\operatorname{Cost}} / \gamma^{\operatorname{Bias}}}.
	\end{equation}
\end{corollary}
Using \Cref{corollary_saob_complexity_optimal} we may bound the complexity of the SAOB by the complexity of MC.
\begin{corollary}[SAOB cost upper bounded by MC] \label{corollary_saob_bound_mc}
Let the assumptions of \Cref{corollary_mc_complexity_rate} be true. Then the $\saob_L$ and $\saobk{q}_L$ with $q \geq 1$ achieve a MSE of $\varepsilon^2$ with a cost not larger than the cost of the MC estimator in \eqref{equation_mc_asymptotic_cost}.
\end{corollary}
\subsection{Multilevel Monte Carlo estimator} \label{subsection_multilevel_monte_carlo}
The MLMC estimator \cite{giles_2008, giles_2015} uses the vector $\alpha^L = e_L$ together with
\begin{equation*}
\begin{aligned}
S^1 &= \{1 \}, &&\quad \beta^1 = e_1, &\\
S^k &= \{k - 1, k\}, &&\quad \beta^k = e_k - e_{k - 1}, &\quad k = 2,\dots,L.
\end{aligned}
\end{equation*}
The estimator reads
\begin{equation*}
\mlmce_L = \sum_{k = 2}^L \frac{1}{m_k} \sum_{i = 1}^{m_k} (Z_k(\omega^k_i) - Z_{k - 1}(\omega^k_i)) + \frac{1}{m_1} \sum_{i = 1}^{m_1} Z_1(\omega^1_i).
\end{equation*} 
We recall the well known MLMC complexity theorem.
\begin{corollary}[MLMC asymptotic cost, {\cite[Theorem 1]{Cliffe_2011}}] \label{corollary_mlmc_complexity_rates}
	Assume that there exist positive constants $\gamma^{\operatorname{Bias}}, \gamma^{\operatorname{Var}}$ and $\gamma^{\operatorname{Cost}}$ such that the following statements hold for all $\ell \in \mathbb{N}$,
	\begin{align}
	\tag{MLMC, M1} \label{equation_M1_mlmc} |\mathbb{E}[Z_\ell] - \mathbb{E}[Z]| \leq c 2^{- \ell \gamma^{\operatorname{Bias}}}, \\
	\tag{MLMC, M2} \label{equation_M2_mlmc} \operatorname{Var}(Z_\ell - Z_{\ell - 1}) \leq c 2^{- \ell \gamma^{\operatorname{Var}}}, \\
	\tag{MLMC, M3} \label{equation_M3_mlmc} \operatorname{Cost}(Z_\ell) \leq c 2^{\ell \gamma^{\operatorname{Cost}}}. 
	\end{align}		
	where the constant $c > 0$ is independent of $\ell$. 
	Then there exists a final level $L > 0$ and numbers of samples $m_1, \dots, m_L$ to achieve a MSE of $\mathbb{E}[|\mlmce_L - \mathbb{E}[Z]|^2] \leq \varepsilon^2$ with costs bounded by
	\begin{equation} \label{equation_mlmc_asymptotic_cost}
	\operatorname{Cost}(\mlmce_L) \leq c \varepsilon^{- \gamma^{\operatorname{Cost}} / \gamma^{\operatorname{Bias}}} + c \begin{cases}
	\varepsilon^{-2}, \quad &\text{if } \gamma^{\operatorname{Cost}} < \gamma^{\operatorname{Var}}, \\
	\varepsilon^{-2} (\log \varepsilon)^2, \quad &\text{if } \gamma^{\operatorname{Cost}} = \gamma^{\operatorname{Var}}, \\
	\varepsilon^{-2 - \frac{\gamma^{\operatorname{Cost}} - \gamma^{\operatorname{Var}}}{\gamma^{\operatorname{Bias}}}}, \quad &\text{if } \gamma^{\operatorname{Cost}} > \gamma^{\operatorname{Var}}.
	\end{cases} 
	\end{equation}
\end{corollary}
Note that the complexity of the MLMC estimator is in general smaller compared to the complexity of the MC estimator.
However, \eqref{equation_M2_mlmc} is a stronger assumption compared to \eqref{equation_M2_mc}. 
Fortunately, this stronger assumption is often satisfied if the model $Z_\ell$ is derived from the discretization of an ODE or PDE and thus satisfies the error estimate
\begin{equation*}
\operatorname{Var}(Z_\ell - Z_{\ell - 1}) \leq \mathbb{E}[(Z_\ell - Z_{\ell - 1})^2] \leq c (\mathbb{E}[(Z_\ell - Z)^2] + \mathbb{E}[(Z_{\ell - 1} - Z)^2]) \leq c 2^{- \ell \gamma^{\operatorname{Var}}},
\end{equation*}
where $\gamma^{\operatorname{Var}} > 0$. We generalize this idea in the next subsection, where we also obtain rates for $\gamma^{\operatorname{Bias}}$ and $\gamma^{\operatorname{Var}}$.
Using again \Cref{corollary_saob_complexity_optimal} gives a complexity bound for the SAOB.
\begin{corollary}[SAOB cost upper bounded by MLMC] \label{corollary_saob_bound_mlmc}
Let the assumptions of \Cref{corollary_mlmc_complexity_rates} be true. Then the $\saob_L$ and $\saobk{q}_L$ with $q \geq 2$ achieve a MSE of $\varepsilon^2$ with a cost not larger than the cost of the MLMC estimator in \eqref{equation_mlmc_asymptotic_cost}.
\end{corollary}
\subsection{Richardson extrapolation estimator} \label{subsection_re_estimators}
Richardson extrapolation (RE) \cite{Asadzadeh_2009,lemaire_2017, Richardson_1911,Romberg_1955} is a well known technique to improve the accuracy of a collection of parametric models . 
The idea is to linearly combine models with a low accuracy, and to increase the convergence rate to the truth $Z$. 
We remark that RE was already used in \cite{giles_2008, lemaire_2017} to increase the bias rate $\gamma^{\operatorname{Bias}}$. 
The analysis in this section is similar to the analysis of MLMC. We replace the assumptions \eqref{equation_M1_mlmc}, \eqref{equation_M2_mlmc} and \eqref{equation_M3_mlmc} with the following assumptions.

\begin{assumption}[Pathwise Asymptotic Expansion] \label{assumption_re_expansion}
\begin{itemize}
\item[$(i)$] There exists $q \in \mathbb{N}$, rates $0 = \gamma^1 < \dots < \gamma^q$ and random variables $c_2,\dots,c_q$ with bounded second moment such that it holds
\begin{equation} \label{equation_re_expansion2}
Z_\ell(\omega) = Z(\omega) + \sum_{j = 2}^{q - 1} c_j(\omega) 2^{- \ell \gamma^j} + \mathcal{O}(2^{- \ell \gamma^q}), \quad \text{as }\ell \rightarrow +\infty.
\end{equation}
\item[$(ii)$] The costs for an evaluation of $Z_\ell$ are bounded by
\begin{equation} \label{equation_re_geometric_cost}
\operatorname{Cost}(Z_\ell) \leq c 2^{\ell \gamma^{\operatorname{Cost}}}.
\end{equation}
\end{itemize}
\end{assumption}
The $\mathcal{O}$-notation in \eqref{equation_re_expansion2} has to be understood in the $L^2$-sense for $\ell \rightarrow +\infty$, that is,
\begin{equation*}
\mathbb{E}\left[\left( Z_\ell(\omega) - Z(\omega) - \sum_{j = 2}^{q - 1} c_j(\omega) 2^{- \ell \gamma^j} \right)^2 \right] \leq c 2^{- 2 \ell \gamma^q} \quad \text{for all } \ell \text{ large enough}.
\end{equation*}
Note that \eqref{equation_re_geometric_cost} is in fact \eqref{equation_M3_mlmc}.
For a fixed event $\omega$ the assumption in \eqref{equation_re_expansion2} is classical in the theory of Richardson extrapolation.
It has been adapted to our setting where we deal with realizations of random variables in contrast to deterministic quantities.

The idea of RE is to linearly combine the models $Z_\ell$ to remove the term $\sum_{j = 2}^q c_j(\omega) 2^{- \ell \gamma^j}$ in \eqref{equation_re_expansion2}. 
To this end we define the RE coefficients
\begin{equation} \label{equation_re_vectors}
v^{k, q} := \begin{cases}
0, &\quad \text{if } k = 0, \\
e_1, &\quad \text{if } k = 1, \\
(2^{\gamma^k} D v^{k - 1, q} - v^{k - 1, q})/(2^{\gamma^k} - 1), &\quad \text{if } 1 < k < q, \\
D v^{k - 1, q}, &\quad \text{if } k \geq q,
\end{cases}
\end{equation}
where $D$ is a down shift matrix
\begin{equation}
D := \begin{pmatrix}
0 & 0 \\
I_{L - 1, L - 1} & 0
\end{pmatrix} \in \mathbb{R}^{L \times L}.
\end{equation}
The idea of the RE estimator is analogous to MLMC in the sense that we also use a telescoping sum.
Let $\alpha^L = v^{L, q}$ and
\begin{equation} \label{equation_re_model_groups}
\begin{aligned}
S^1 &= \{1 \}, &&\; \beta^1 = v^{1, q}, &\\
S^k &= \{\max\{k - q + 1, 1\}, \max\{k - q + 2, 1\}, \dots , k\}, &&\; \beta^k = v^{k, q} - v^{k - 1, q}, &\quad k = 2, \dots, L.
\end{aligned}
\end{equation}
The $\operatorname{RE,q}$ estimator is then defined as
\begin{equation*}
\ree{q}_{v^{L, q}} = \sum_{k = 1}^L \sum_{\ell \in S^k} (v^{k, q}_\ell - v^{k - 1, q}_\ell) \frac{1}{m_k} \sum_{i = 1}^{m_k} Z_\ell(\omega^k_i).
\end{equation*} 
A close inspection shows that for $q=2$ this estimator is actually the MLMC estimator. 
The $\operatorname{RE}, q$ couples at most $q$ models such that $|S^k| \leq q$ for $k = 1, \dots, L$. 
We have the following result.
\begin{lemma} \label{lemma_re_bias_variance}
Let \Cref{assumption_re_expansion} be true. 
Then the following estimates hold
\begin{align}
\tag{RE, M1} \label{equation_M1_re} |(v^{k, q})^T \mu - \mathbb{E}[Z]| &\leq c 2^{-k \gamma^q}, \\
\tag{RE, M2} \label{equation_M2_re} (v^{k, q} - v^{k - 1, q})^T C (v^{k, q} - v^{k - 1, q}) &\leq c 2^{-2 k \gamma^q}, \\
\tag{RE, M3} \label{equation_M3_re} W_k &\leq c 2^{k \gamma^{\operatorname{Cost}}}.
\end{align}		
\end{lemma}
\begin{proof}
See \Cref{appendix_proof_lemma_re_bias_variance}.
\end{proof}

We now state complexity rates for the RE estimator which can be proved using \Cref{lemma_re_bias_variance} and standard arguments.
\begin{corollary}[RE asymptotic cost] \label{corollary_re_complexity_rates}
	Let \Cref{assumption_re_expansion} or \eqref{equation_M1_re}, \eqref{equation_M2_re} and \eqref{equation_M3_re} be true. 
	Then there exists a final level $L > 0$ and numbers of samples $m_1, \dots, m_L$ to achieve a MSE of $\mathbb{E}[|\ree{q}_{v^{L, q}} - \mathbb{E}[Z]|^2] \leq \varepsilon^2$ with costs bounded by
	\begin{equation} \label{equation_re_asymptotic_cost}
	\operatorname{Cost}(\ree{q}_{v^{L, q}}) \leq c \varepsilon^{- \gamma^{\operatorname{Cost}} / \gamma^q} + c \begin{cases}
	\varepsilon^{-2}, \quad &\text{if } \gamma^{\operatorname{Cost}} < 2 \gamma^q, \\
	\varepsilon^{-2} (\log \varepsilon)^2, \quad &\text{if } \gamma^{\operatorname{Cost}} = 2 \gamma^q, \\
	\varepsilon^{-2 - \frac{\gamma^{\operatorname{Cost}} - 2 \gamma^q}{\gamma^q}}, \quad &\text{if } \gamma^{\operatorname{Cost}} > 2 \gamma^q.
	\end{cases} 
	\end{equation}
\end{corollary}
Observe that the complexity of the RE estimator is improved compared to the MLMC estimator since we can leverage the additional smoothness of the models $Z_\ell$ in assumption \eqref{equation_re_expansion2}. 
We obtain an increased bias rate of $\gamma^q$ instead of $\gamma^2$ and an increased variance rate of $2 \gamma^q$ instead of $2 \gamma^2$.  
Using \Cref{corollary_saob_complexity_optimal} we obtain another complexity bound for the SAOB.
\begin{corollary}[SAOB cost upper bounded by RE] \label{corollary_saob_bound_re}
Let the assumptions of \Cref{corollary_re_complexity_rates} be true. Then the $\saob_{v^{L, q}}$ and $\saobk{s}_{v^{L, q}}$ with $s \geq q$ achieve a MSE of $\varepsilon^2$ with a cost not larger than the cost of the RE estimator in \eqref{equation_re_asymptotic_cost}.
\end{corollary}
\subsection{Weighted Richardson extrapolation estimator}
Finally, we generalize the RE estimators to obtain different bias and variance rates. 
Note that for $s \in \mathbb{N}$ the definition \eqref{equation_re_vectors} shows that the vectors
\begin{equation*}
\{v^{1, s} - v^{0, s}, v^{2, s} - v^{1, s}, \dots, v^{L, s} - v^{L - 1, s}\}
\end{equation*}
form a basis of $\mathbb{R}^L$. 
Thus there exist weights $a_1, \dots, a_L$ such that for $t \in \mathbb{N}$ it holds
\begin{equation} \label{equation_re_basis}
v^{L, t} = \sum_{k = 1}^L a_k (v^{k, s} - v^{k - 1, s}).
\end{equation}
Now, we define the \textit{weighted RE estimator} 
\begin{equation*}
\ree{s}_{v^{L, t}} = \sum_{k = 1}^L a_k \sum_{\ell \in S^k } (v^{k, s}_\ell - v^{k - 1, s}_\ell) \frac{1}{m_\ell} \sum_{i = 1}^{m_\ell} Z_\ell(\omega^k_i),
\end{equation*}
which is a linear unbiased estimator for $(v^{L, t})^T \mu$. 
A result similar to \Cref{lemma_re_bias_variance} holds.
\begin{lemma} \label{lemma_weighted_re_bias_variance}
Let \Cref{assumption_re_expansion} be true and $t, s \leq q$. 
Then the following estimates hold:
\begin{align}
\tag{WRE, M1} \label{equation_M1_wre} |(v^{\ell, t})^T \mu - \mathbb{E}[Z]| &\leq c 2^{-\ell \gamma^t}, \\
\tag{WRE, M2} \label{equation_M2_wre} [a_k (v^{k, s} - v^{k - 1, s})]^T C [a_k (v^{k, s} - v^{k - 1, s})] &\leq c a^2_k  2^{-2 k \gamma^s}, \\
\tag{WRE, M3} \label{equation_M3_wre} W_k &\leq c 2^{k \gamma^{\operatorname{Cost}}}.
\end{align}		
Furthermore, if $s \leq t$ then $|a_k| \leq c$ with a constant $c$ independent of $k$. 
\end{lemma}

\begin{proof}
See  \Cref{appendix_proof_lemma_weighted_re_bias_variance}.
\end{proof}

Note that in numerical experiments we observed that the property $|a_k| \leq c$ with a constant $c$ independent of $k$ also holds for $s > t$, however, we do not see a way to prove it. 
We directly obtain the analogous result of \Cref{corollary_re_complexity_rates}.
\begin{corollary}[Weighted RE asymptotic cost] \label{corollary_weighted_re_complexity_rates}
	Let \Cref{assumption_re_expansion} or \eqref{equation_M1_wre}, \eqref{equation_M2_wre} and \eqref{equation_M3_wre} be true with $s \leq t \leq q$. 
	Then there exists a final level $L > 0$ and numbers of samples $m_1, \dots, m_L$ to achieve a MSE of $\mathbb{E}[|\ree{s}_{v^{L, t}} - \mathbb{E}[Z]|^2] \leq \varepsilon^2$ with costs bounded by
	\begin{equation} \label{equation_weighted_re_asymptotic_cost}
	\operatorname{Cost}(\ree{s}_{v^{L, t}}) \leq c \varepsilon^{- \gamma^{\operatorname{Cost}} / \gamma^t} + c \begin{cases}
	\varepsilon^{-2}, \quad &\text{if } \gamma^{\operatorname{Cost}} < 2 \gamma^s, \\
	\varepsilon^{-2} (\log \varepsilon)^2, \quad &\text{if } \gamma^{\operatorname{Cost}} = 2 \gamma^s, \\
	\varepsilon^{-2 - \frac{\gamma^{\operatorname{Cost}} - 2 \gamma^s}{\gamma^t}}, \quad &\text{if } \gamma^{\operatorname{Cost}} > 2 \gamma^s.
	\end{cases} 
	\end{equation}
\end{corollary}
The smallest asymptotic cost is achieved if $s = t = q$ with maximal $q$ such that \Cref{assumption_re_expansion} holds. 
In this case \Cref{corollary_weighted_re_complexity_rates} is identical to \Cref{corollary_re_complexity_rates}. 
The complexity bound for the SAOB follows again from \Cref{corollary_saob_complexity_optimal}.
\begin{corollary}[SAOB cost upper bounded by weighted RE] \label{corollary_saob_bound_weighted_re}
Let the assumptions of \Cref{corollary_weighted_re_complexity_rates} be true. Then the $\saob_{v^{L, t}}$ and $\saobk{s^\prime}_{v^{L, t}}$ with $s^\prime \geq s$ achieve a MSE of $\varepsilon^2$ with a cost not larger than the cost of the weighted RE estimator in \eqref{equation_weighted_re_asymptotic_cost}.
\end{corollary}
\section{Numerical experiments with a PDE in 2D space}\label{sec:numerics} 
In this section we verify the main results of this paper numerically. 
We revisit the example in \cite[Sec. 6.3]{SchadenUllmann:2020}. 
Recall that the random variable $Z$ is the domain average with realizations
\begin{equation}\label{qoi_elliptic}
Z(\omega) := \frac{1}{|D_{obs}|} \int_{D_{obs}} y(x, \omega) dx,
\end{equation}
where $D_{obs} := (3 / 4, 7 / 8) \times (7 / 8, 1)$ is a subset of the unit square domain $D := (0, 1)^2$. 
The function $y$ is a weak solution of the elliptic boundary value problem whose pathwise, strong form reads 
\begin{equation} \label{equation_elliptic_pde}
\begin{aligned}
- \operatorname{div}(a(x, \omega) \nabla y(x, \omega)) &= 1, \quad x \in D, \\
y(x, \omega) &= 0, \quad x \in \partial D
\end{aligned}
\end{equation}
for almost all $\omega$.
The diffusion coefficient is $a(x, \omega) := \exp(b(x, \omega))$, where $b$ is a mean zero Gaussian random field with Whittle--Mat\'ern covariance function \cite{Stein_1999} with smoothness parameter $\nu= 3 / 2$, variance $\sigma^2 = 1$, and correlation length $\rho = 0.5$. 
To discretize $Z$ we use a standard piecewise linear finite element (FE) discretization of \eqref{equation_elliptic_pde} with $L = 6$ different levels of accuracy. 
We obtain $Z_\ell$ by uniform refinement of the previous mesh, starting with a regular grid for $Z_1$. We define the bias as follows:
\begin{equation}\label{bias}
\operatorname{Bias}(Z_\ell) := |\mathbb{E}[Z_\ell] - \mathbb{E}[Z]|.
\end{equation}
The discretization data is summarized in \Cref{table_elliptic_data}. 
We use $10^5$ pilot samples to estimate the $\operatorname{Bias}$ in \eqref{bias}, the model covariance matrix $C$, and the average cost $w_\ell$ to compute a sample of $Z_\ell$. 
The cost for the pilot samples is not included in the subsequent analysis.
The fourth column in \Cref{table_elliptic_data} suggests the cost increase rate $\gamma^{\operatorname{Cost}} = 2$ which corresponds to a four-fold cost increase as expected in 2D space.

\begin{table} 
	\begin{center}
		\begin{tabular}{| c | r  c  c  c  c |}
			\hline
			Model & \#Nodes & $w_\ell$ & $w_\ell / w_{\ell - 1}$ & $\operatorname{Var}(Z_\ell)$ & $\operatorname{Bias}(Z_\ell)$ \\\hline
			$Z_1$ & 81 & 0.0019s & -- & $1.52 \cdot 10^{-2}$ &  $3.80 \cdot 10^{-3}$ \\ 
			$Z_2$ & 289 & 0.0025s & 1.33 & $2.35 \cdot 10^{-2}$  & $1.02 \cdot 10^{-3}$ \\ 
			$Z_3$ & 1089 & 0.0050s & 2.00 & $2.62 \cdot 10^{-2}$  &  $2.58 \cdot 10^{-4}$ \\  
			$Z_4$ & 4225 & 0.0185s & 3.70 & $2.69 \cdot 10^{-2}$  & $6.17 \cdot 10^{-5}$ \\ 
			$Z_5$ & 16641 & 0.0736s & 3.98 & $2.71 \cdot 10^{-2}$ & $1.24 \cdot 10^{-5}$ \\ 
			$Z_6$ & 66049 & 0.3149s & 4.28 & $2.71 \cdot 10^{-2}$ & $2.98 \cdot 10^{-6}$ \\ \hline 
		\end{tabular}
	\end{center}
	\caption{The number of finite element basis functions (\#Nodes) characterizes the discretization. $w_\ell$ is the average time to compute a realization of $Z_\ell$, and the cost increase factor is denoted by $w_\ell / w_{\ell - 1}$. The last two columns contain the estimates for the variance and bias.} \label{table_elliptic_data}
\end{table}

Since the random field $b$ is smooth and we integrate over $y$ in \eqref{qoi_elliptic}, we expect a RE expansion of the form
\begin{equation} \label{equation_re_numerical_guess}
Z_\ell(\omega) = Z(\omega) + c_2(\omega) 2^{- 2 \ell} + c_3(\omega) 2^{- 4 \ell} + o(2^{-4 \ell}) \quad \text{as } \ell \rightarrow +\infty.
\end{equation} 
If the expansion in \eqref{equation_re_numerical_guess} holds, we have $q = 3$ with rates $\gamma^1 = 0$, $\gamma^2 = 2$ and $\gamma^3 = 4$. 
Therefore we expect to observe the following rates according to \Cref{lemma_re_bias_variance}:
\begin{equation} \label{equation_re_numerical_guess_rates}
\begin{aligned}
\operatorname{Bias}(Z_\ell) \simeq 2^{- 2 \ell}, & \quad &\operatorname{Var}(Z_\ell - Z_{\ell - 1}) \simeq 2^{- 4 \ell}, \\
\operatorname{Bias}((v^{\ell, 3})^T (Z_j)_{j = 1}^\ell) \simeq 2^{- 4 \ell}, & \quad &\operatorname{Var}((v^{\ell, 3} - v^{\ell - 1, 3})^T (Z_j)_{j = 1}^\ell) \simeq 2^{- 8 \ell}
\end{aligned}
\end{equation}
and $\operatorname{Var}(Z_\ell) \simeq 1$. 
The notation $\phi_1(\ell) \simeq \phi_2(\ell)$ means that there exist constants $c_1, c_2 > 0$ such that
\begin{equation*}
\phi_1(\ell) \leq c_1 \phi_2(\ell), \quad \phi_2(\ell) \leq c_2 \phi_1(\ell), \quad \text{for all } \ell \text{ sufficiently large}.
\end{equation*}
\Cref{corollary_re_complexity_rates} can be applied if \eqref{equation_re_numerical_guess_rates} is true, which we assume in the remainder of this section. We provide a numerical verification of the rates in \Cref{figure_bias_variance_rates_numerical_guess}.
A formal proof of the rates in \eqref{equation_re_numerical_guess_rates} or the expansion in \eqref{equation_re_numerical_guess} is beyond the scope of this paper.
\begin{figure} 
	\begin{center}
		\includegraphics[width=1.0 \textwidth]{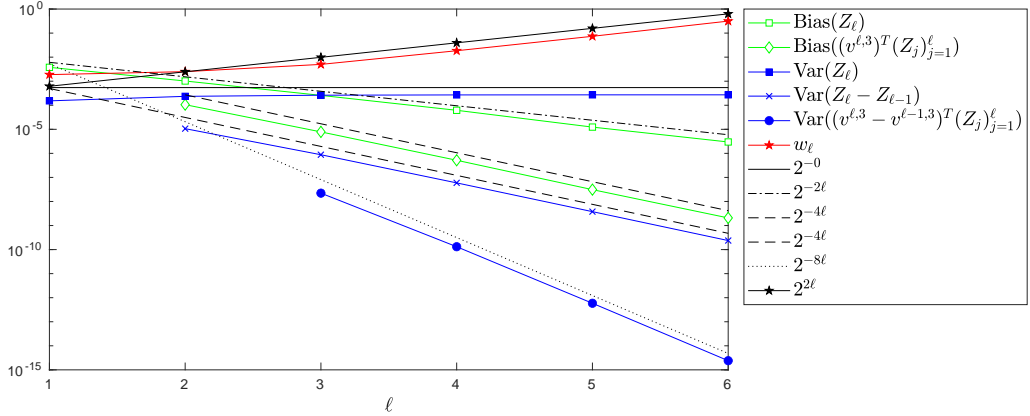}
	\end{center}
	\caption{Biases, variances and work w.r.t. the refinement level $\ell$.} 
	\label{figure_bias_variance_rates_numerical_guess}
\end{figure}

\subsection{Complexity of estimators}\label{comp_ell}
We now study the computational complexity of various unbiased estimators to approximate $\mu_\ell = \mathbb{E}[Z_\ell]$. First, we consider Monte Carlo $\mce_\ell$, Multilevel Monte Carlo $\mlmce_\ell$ and the SAOB $\saobk{2}_\ell$, $\saobk{3}_\ell$, and $\saob_\ell$ with coupling numbers $q = 2$, $q = 3$ and $q = +\infty$, respectively. We also consider the Multifidelity Monte Carlo estimator $\mfmce_\ell$ in \cite{Peherstorfer_2016}.

We measure the accuracy of $\widehat{\mu}_\ell$, an unbiased estimator of $\mu_\ell=\mathbb{E}[Z_\ell]$, by the MSE
\begin{equation}\label{mse}
\mathbb{E}[(\widehat{\mu}_\ell - \mathbb{E}[Z])^2] = \operatorname{Bias}(Z_\ell)^2 + \operatorname{Var}(\widehat{\mu}_\ell).
\end{equation}
For each level $\ell=1,\dots,L$ we ensure that $\operatorname{Bias}(Z_\ell)^2 = \operatorname{Var}(\widehat{\mu}_\ell)$ resulting in a MSE equal to $2 \cdot \operatorname{Bias}(Z_\ell)^2$. Let $\varepsilon > 0$ denote a given tolerance. To achieve a MSE $\leq \varepsilon^2$ in \eqref{mse} all estimators in this section have computational costs bounded by an expression of the form
\begin{equation}\label{equation_generic_cost}
\text{Cost}(\widehat\mu_\ell)
\leq \phi(\varepsilon) = 
c \varepsilon^{- \gamma^{\operatorname{Cost}} / \gamma^{\operatorname{Bias}}} + c \begin{cases}
\varepsilon^{-2}, \quad &\text{if } \gamma^{\operatorname{Cost}} < \gamma^{\operatorname{Var}}, \\
\varepsilon^{-2} (\log \varepsilon)^2, \quad &\text{if } \gamma^{\operatorname{Cost}} = \gamma^{\operatorname{Var}}, \\
\varepsilon^{- 2 - ({\gamma^{\operatorname{Cost}} -  \gamma^{\operatorname{Var}}})/\gamma^{\operatorname{Bias}}}, \quad &\text{if } \gamma^{\operatorname{Cost}} > \gamma^{\operatorname{Var}}.
\end{cases}
\end{equation}
Note that the cost bound $\phi(\varepsilon)$ in \eqref{equation_generic_cost} consists of two terms.
The first summand arises from ceiling the optimal, but possibly fractional number of samples, and the second summand arises from the constraint MSE $\leq \varepsilon^2$.

In addition to the unbiased estimators for $\mu_\ell$ we also consider $\ree{3}_{v^{\ell, 3}}$, $\saobk{2}_{v^{\ell, 3}}$, $\saobk{3}_{v^{\ell, 3}}$ and $\saob_{v^{\ell, 3}}$, which are unbiased estimators for $(v^{\ell, 3})^T \mu$. 
We record the complexity bounds for the estimators in \Cref{table_elliptic_complexity}. For $\gamma^{\operatorname{Cost}} = 2$ we observe that all estimators have the optimal complexity $\varepsilon^{-2}$ except the MC estimator with complexity $\varepsilon^{-3}$. 
The complexity bounds are verified numerically in \Cref{figure_complexity_elliptic_truth} where we plot the \textit{computed} costs. 
We generated this plot by first computing the optimal fractional number of samples for each estimator and then ceiling this number. 

A more interesting setting is achieved by using the artificial, increased cost per sample
\begin{equation} \label{equation_theoretical_costs}
\operatorname{Cost}(Z_\ell) = 10^{-6} \cdot 2^{6 \ell},
\end{equation}
which corresponds to a rate of $\gamma^{\operatorname{Cost}} = 6$. We list the asymptotic upper bound for the cost $\phi$ of all estimators in \Cref{table_elliptic_complexity}. 
Note that the bounds for MC, MLMC and the RE estimators follows from \Cref{section_upper_cost_bound}.
The bound for MFMC is proved in \cite{Peherstorfer_2018}. 
For the $\operatorname{SAOB}$ we use the procedure outlined at the beginning of \Cref{section_upper_cost_bound}, that is, we consider a sequence of estimators whose complexity is an upper bound for the SAOBs.
In particular, for $\saob_L$ we obtain the rate $\varepsilon^{-3}$ by comparison with $\mlmce_L$, and for $\saob_{v^{L, 3}}$ we obtain the rate $\varepsilon^{-2}$ by comparison with $\ree{3}_{v^{L, 3}}$. \Cref{remark_saob_kappa_complexity} shows that $\saobk{q}_\alpha$ has equal or smaller complexity compared to $\ree{q}_\alpha$ and thus we use the bound for the latter. The special case $\ree{2}_L = \mlmce_L$ shows that $\saobk{q}_L$ has costs not exceeding the costs of MLMC for $q \geq 2$. 
\par
The resulting costs of the estimators are plotted in \Cref{figure_complexity_elliptic_artificial}, thereby confirming the claims made in \Cref{table_elliptic_complexity}. 
Note that the cost for $\mlmce_L$, $\saobk{3}_L$ and $\saob_L$ is of order $\varepsilon^{-3}$ and hence suboptimal.
However, for $\saobk{3}_L$ and $\saob_L$ the suboptimal cost is a result of ceiling the number of samples: Since at least one high fidelity model has to be evaluated, the total cost is lower bounded by $\varepsilon^{-3}$.
Without ceiling, the total cost would in fact be optimal with an order of $\varepsilon^{-2}$.
This is not the case for $\mlmce_L$ where the suboptimal cost is of order $\varepsilon^{-3}$ with or without ceiling the number of samples.


In \Cref{figure_complexity_elliptic_artificial} we also plot $\operatorname{SAOB}$(*), which is the $\operatorname{SAOB}$ without ceiling the number of samples. 
This estimator requires the evaluation of $\approx 5 \cdot 10^{-4}$ samples of the high-fidelity model which is not a natural number and thus impossible in practice. This shows that the complexity is $\varepsilon^{-2}$ if we drop the term $\varepsilon^{- \gamma^{\operatorname{Cost}} / \gamma^{\operatorname{Bias}}} = \varepsilon^{-3}$ associated with ceiling. 
Importantly, we can achieve the optimal complexity by using the bias vector $v^{\ell, 3}$ instead of $v^{\ell, 2} = e_\ell$, which is shown in the right image of \Cref{figure_complexity_elliptic_artificial}. 
This change improves the overall complexity to $\varepsilon^{-2}$, since the rounding cost does not dominate.
Finally, the $\saobk{2}_{v^{\ell, 3}}$ does not couple three models and  thus only achieves a rate of $\varepsilon^{-2.5}$, which is a consequence of the small rate $\gamma^{\operatorname{Var}} = 4$. 
\renewcommand{\arraystretch}{1.2}
\begin{table} 
	\begin{center}
		\begin{tabular}{|l|c | cc|c |c| l |}\hline
			Estimator  &  &  &  & True $\gamma^{\text{Cost}}=2$ & Artificial $\gamma^{\text{Cost}}=6$ &  Justification \\	 
			& $q$ & $\gamma^{\text{Bias}}$ & $\gamma^{\text{Var}}$ & $\phi(\varepsilon)$ & $\phi(\varepsilon)$ &  \\ \hline  
			$\mce_L$   & 1 & 2 & 0 & $\varepsilon^{-1} + \varepsilon^{-3}$ & $\varepsilon^{-3} + \varepsilon^{-5}$ & \Cref{corollary_mc_complexity_rate} \\
			$\mlmce_L$ & 2 & 2 & 4 & $\varepsilon^{-1} + \varepsilon^{-2}$ & $\varepsilon^{-3} + \varepsilon^{-3}$ & \Cref{corollary_mlmc_complexity_rates} \\
			$\mfmce_L$ & $L$ & 2 & 4 & $\varepsilon^{-1} + \varepsilon^{-2}$ & $\varepsilon^{-3} + \varepsilon^{-3}$ & \cite{Peherstorfer_2018} \\ \hline
			$\saobk{2}_L$ & 2 & 2 & 4 & $\varepsilon^{-1} + \varepsilon^{-2}$ & $\varepsilon^{-3} + \varepsilon^{-3}$ & \Cref{corollary_saob_bound_mlmc} \\ 
			$\saobk{3}_L$ & 3 & 2 & 8 & $\varepsilon^{-1} + \varepsilon^{-2}$ & $\varepsilon^{-3} + \varepsilon^{-2}$ & \Cref{corollary_saob_bound_weighted_re} (*) \\
			$\saob_L$ & $L$ & 2 & 8 & $\varepsilon^{-1} + \varepsilon^{-2}$ & $\varepsilon^{-3} + \varepsilon^{-2}$ & \Cref{corollary_saob_bound_weighted_re} (*) \\ \hline
			$\mce_{v^{L, 3}}$ & 1 & 4 & 0 & $\varepsilon^{-0.5} + \varepsilon^{-2.5}$ & $\varepsilon^{-1.5} + \varepsilon^{-3.5}$ &  \Cref{corollary_mc_complexity_rate} \\
			$\ree{2}_{v^{L, 3}}$ & 2 & 4 & 4 & $\varepsilon^{-0.5} + \varepsilon^{-2}$ & $\varepsilon^{-1.5} + \varepsilon^{-2.5}$ & \Cref{corollary_weighted_re_complexity_rates} \\
			$\ree{3}_{v^{L, 3}}$ & 3 & 4 & 8 & $\varepsilon^{-0.5} + \varepsilon^{-2}$ & $\varepsilon^{-1.5} + \varepsilon^{-2}$ & \Cref{corollary_re_complexity_rates} \\ \hline
			$\saobk{2}_{v^{L, 3}}$ & 2 & 4 & 4 & $\varepsilon^{-0.5} + \varepsilon^{-2}$ & $\varepsilon^{-1.5} + \varepsilon^{-2.5}$ & \Cref{corollary_saob_bound_weighted_re} \\ 
			$\saobk{3}_{v^{L, 3}}$ & 3 & 4 & 8 & $\varepsilon^{-0.5} + \varepsilon^{-2}$ & $\varepsilon^{-1.5} + \varepsilon^{-2}$ & \Cref{corollary_saob_bound_re} \\ 
			$\saob_{v^{L, 3}}$ & $L$ & 4 & 8 & $\varepsilon^{-0.5} + \varepsilon^{-2}$ & $\varepsilon^{-1.5} + \varepsilon^{-2}$ & \Cref{corollary_saob_bound_re} \\ \hline
		\end{tabular}
	\end{center}
	\caption{
		Bias, variance, and cost rates together with the resulting complexity bounds in \eqref{equation_generic_cost} to achieve a MSE of order $\varepsilon^2$. 
		The cost bound $\phi$ is given as order of $\varepsilon$, where the left or right term is associated with the corresponding term of \eqref{equation_generic_cost}. 
		The coupling number $q$ denotes the maximal number of models that are evaluated using the same event $\omega$. 
		The justification marked with (*) requires \Cref{corollary_weighted_re_complexity_rates} and \Cref{corollary_saob_bound_weighted_re} to hold with $s > t$, which we did not prove but we conjecture that it is true. 
		Alternatively, we obtain the exact same upper bound for the total complexity, namely $\varepsilon^{-2}$ for $\gamma^{\text{Cost}}=2$ and $\varepsilon^{-3}$ for $\gamma^{\text{Cost}}=6$ by using the bounds $\varepsilon^{-1}+\varepsilon^{-2}$ for $\gamma^{\text{Cost}}=2$ and $\varepsilon^{-3}+\varepsilon^{-3}$ for $\gamma^{\text{Cost}}=6$ which follow from \Cref{corollary_saob_bound_mlmc}.
	} \label{table_elliptic_complexity}
\end{table}

\begin{figure} 
	\begin{center}
	\includegraphics[trim=8 0 30 10, clip, width=0.49 \textwidth]{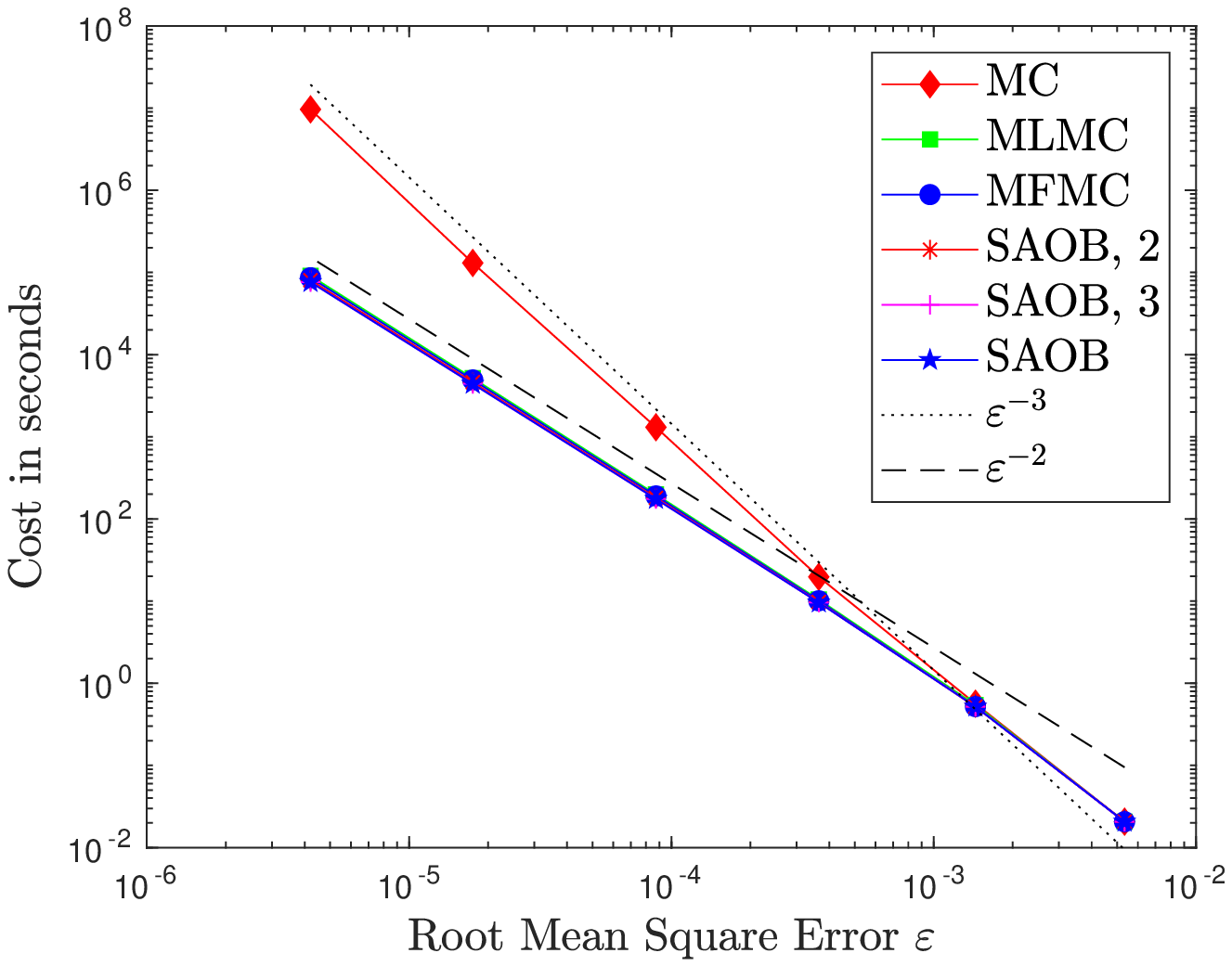}
	\includegraphics[trim=8 0 30 10, clip, width=0.49 \textwidth]{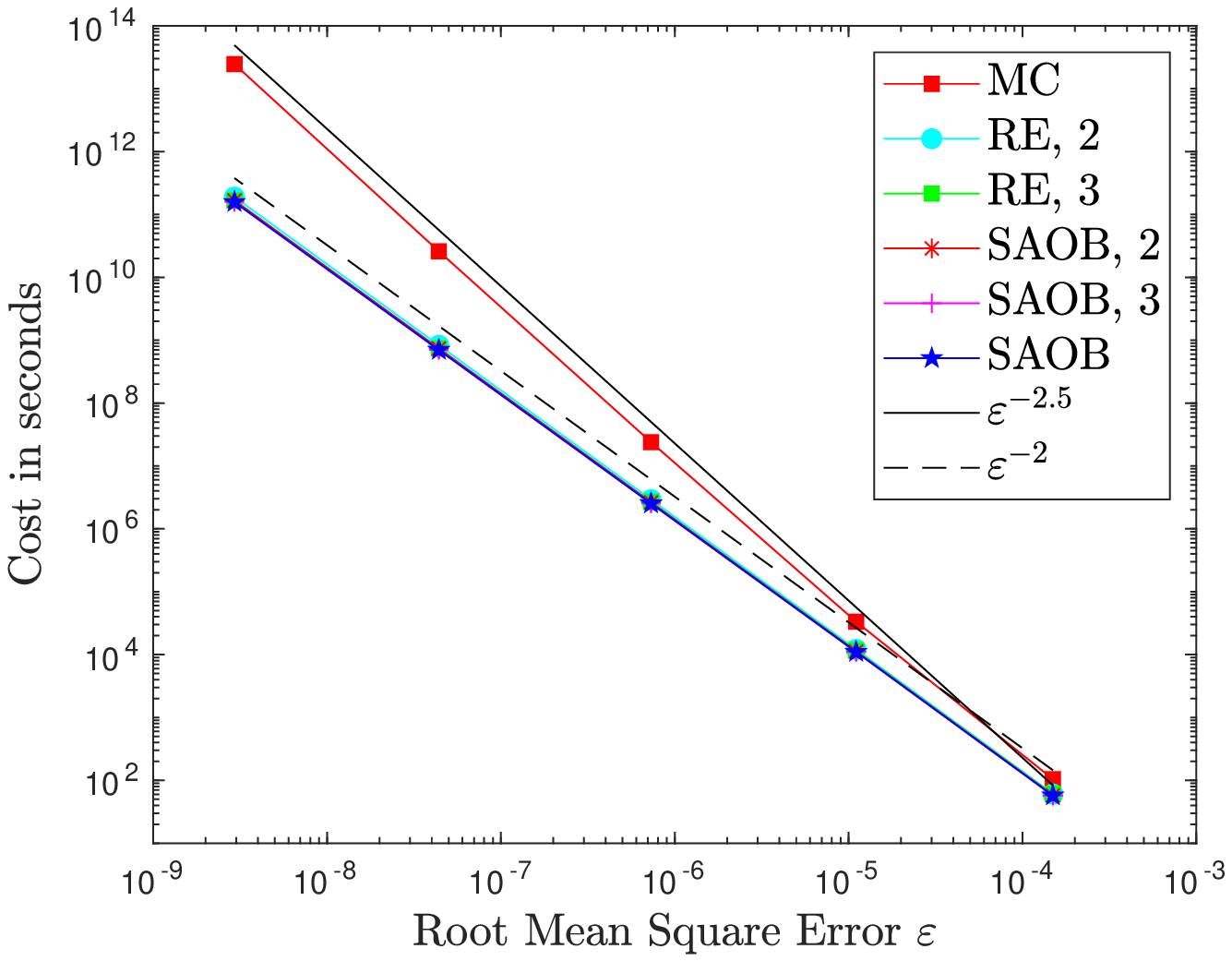}
	\end{center}
	\caption{True cost $\gamma^{\text{Cost}}=2$: unbiased estimators of $\mu_\ell$ on the left and $(v^{\ell, 3})^T \mu$ on the right. Note that the MC estimator on the left is not equal to the MC estimator on the right due to the different bias. The RE, 2 estimator on the right is the weighted RE estimator $\ree{2}_{v^{\ell, 3}}$.}
		 \label{figure_complexity_elliptic_truth}
\end{figure}

\begin{figure} 
	\begin{center}
	\includegraphics[trim=8 0 30 10, clip, width=0.49 \textwidth]{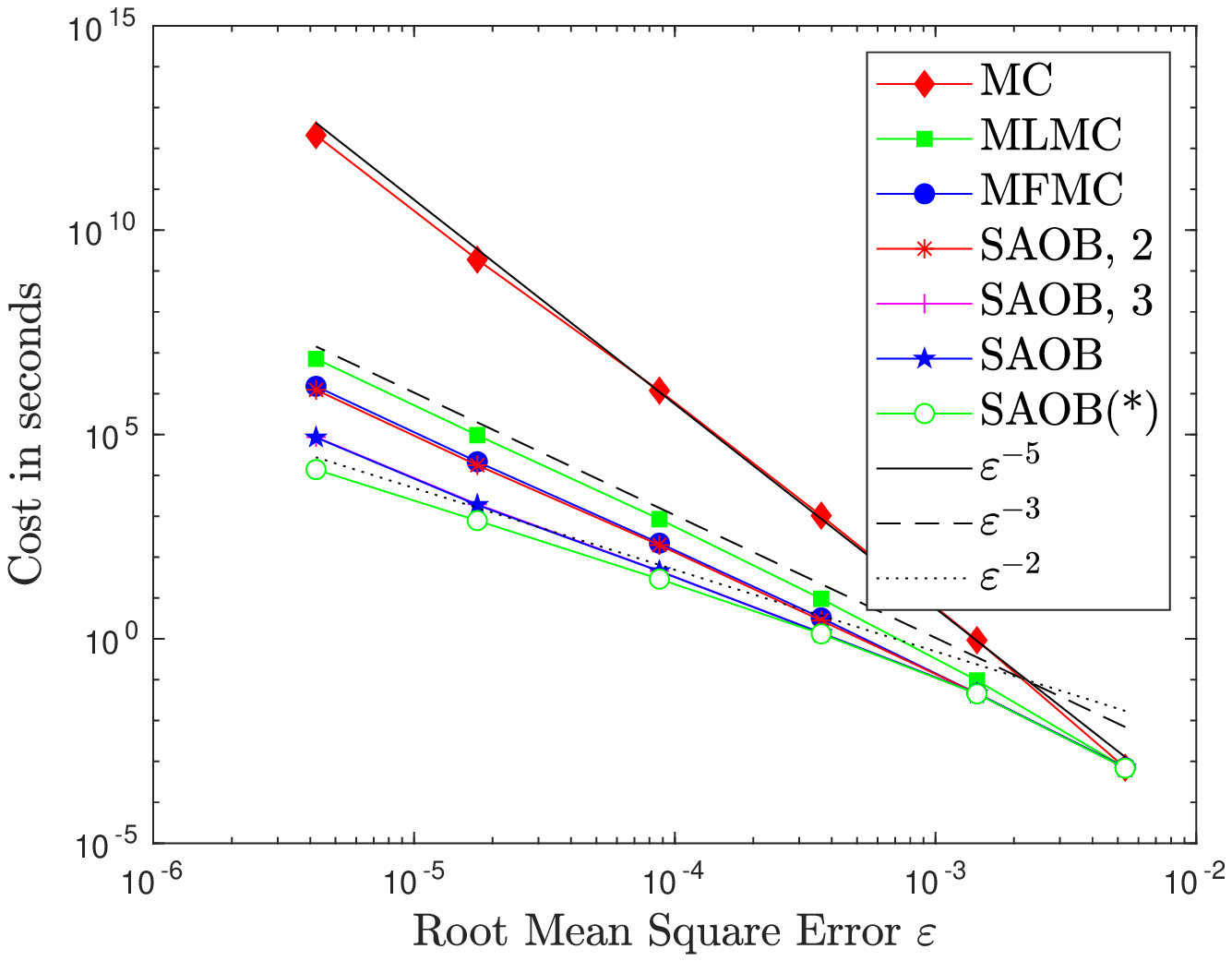}
	\includegraphics[trim=8 0 30 10, clip, width=0.49 \textwidth]{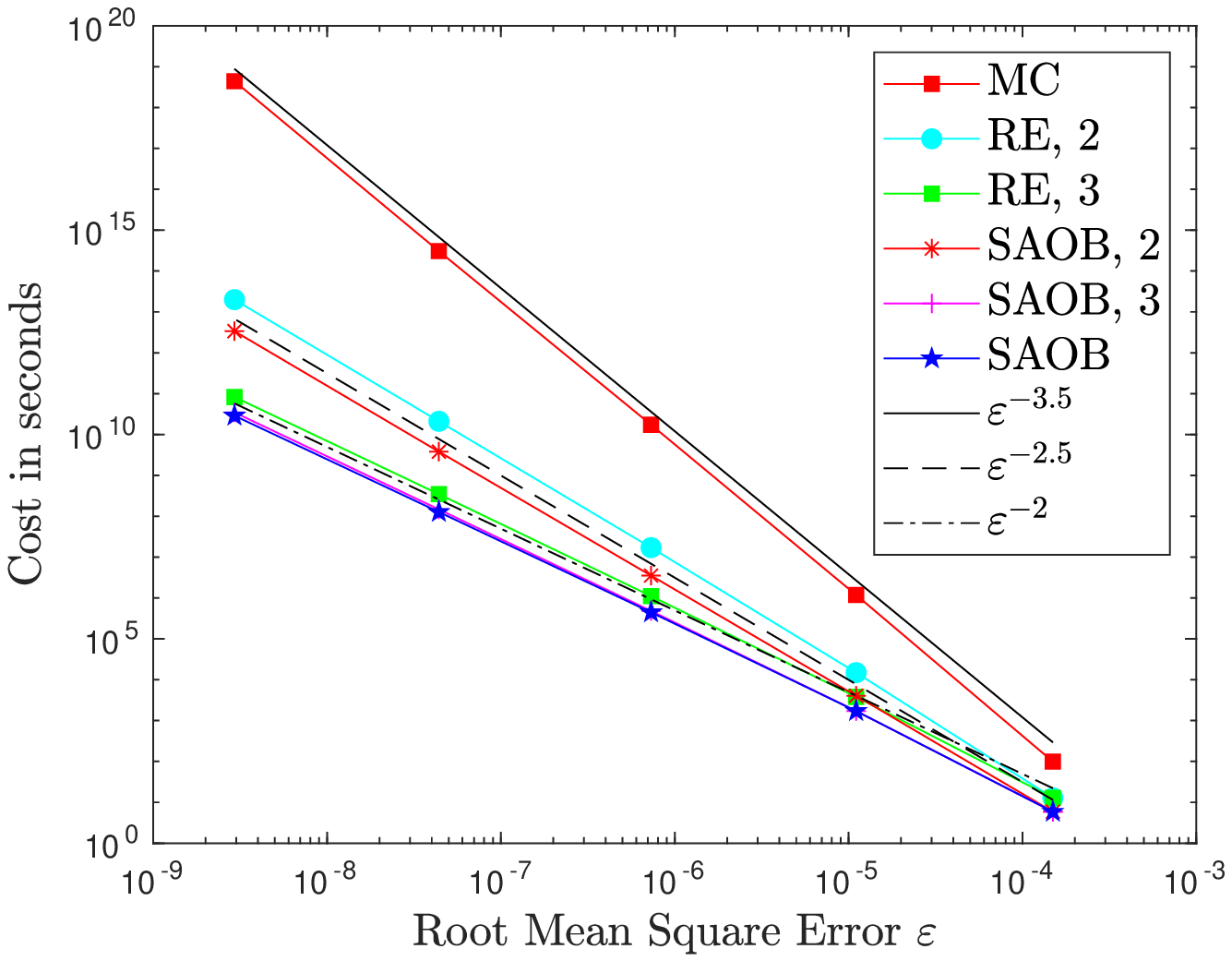}
	\end{center}
	\caption{Artificial cost $\gamma^{\text{Cost}} = 6$: unbiased estimators of $\mu_\ell$ on the left and $(v^{\ell, 3})^T \mu$ on the right. Note that the MC estimator on the left is not equal to the MC estimator on the right due to the different bias. The RE, 2 estimator on the right is the weighted RE estimator $\ree{2}_{v^{\ell, 3}}$. $\operatorname{SAOB}, 3$ and $\operatorname{SAOB}$ have almost identical cost. The costs for $\operatorname{SAOB}$(*) on the left is computed without ceiling the number of samples. }
		 \label{figure_complexity_elliptic_artificial}
\end{figure}

\begin{remark}
Let us describe an informal derivation of the rates $\gamma^{\operatorname{Bias}}$ and $\gamma^{\operatorname{Var}}$ for \eqref{equation_generic_cost} in \Cref{table_elliptic_complexity}. 
For MC, MLMC and the RE estimators the rates follow from \Cref{section_upper_cost_bound} assuming the rates in \eqref{equation_re_numerical_guess_rates} are true. 
For MFMC the authors of \cite{Peherstorfer_2018} show that the rates are the same as for MLMC. 
For the SAOB the rate $\gamma^{\operatorname{Bias}}$ is equal to the bias rate for the RE estimators. 
The rate $\gamma^{\operatorname{Var}}$ follows from the coupling number $q$, i.e. using a single model results in $\gamma^{\operatorname{Var}} = 0$, coupling two models gives $\gamma^{\operatorname{Var}} = 4$ and coupling three or more models gives $\gamma^{\operatorname{Var}} = 8$. 
This reasoning is valid for all examined estimators except for MFMC. 
\end{remark}
Notice that we can only prove upper bounds on the cost of $\saobk{q}_\alpha$.
However, our numerical experiment suggests that these cost bounds are sharp in some cases. 
We examine this finding from a different angle in the next section.
\section{Convergence of SAOB to the RE estimator} \label{section_re_convergence_to_saob}
In this section we conduct a numerical experiment to show that in some cases the coefficients $\beta^k$ for the SAOB converge to the coefficients of RE estimators. Recall that if the aforementioned estimators use the same model groups $S^k$ they can be expressed as 
\begin{equation*}
\widehat{\mu}_L = \sum_{k = 1}^L \sum_{\ell \in S^k} \beta^k_\ell \frac{1}{m_k} \sum_{i = 1}^{m_k} Z_\ell(\omega^k_i),
\end{equation*}
where the coefficients $\beta^k$ depend on the estimator. 
We compare unbiased estimators for $\mathbb{E}[Z_L]$ and thus $e_L = \sum_{k = 1}^L \beta^k$. 
For $\operatorname{SAOB}, q$ and $\operatorname{RE}, q$ with $q = 2, 3, 4$ we define the difference in the coefficients $\beta^k$ as follows
\begin{equation} \label{equation_coefficient_error}
r^{q}(\ell_0) = \left(\sum_{k = 1}^L \|\beta^{k, \operatorname{SAOB, q}}(\ell_0) - \beta^{k, \operatorname{RE}, q} \|^2 \right)^{1 / 2}, 
\end{equation}
where $\ell_0$ is a discretization constant determining the initial accuracy.
We further examine the relative loss of the variance using the RE estimator instead of the optimal $\operatorname{SAOB}$.
To this end we define
\begin{equation} \label{equation_variance_error}
e^{q}(\ell_0) = \frac{\operatorname{Var}(\ree{q}_L(\ell_0)) - \operatorname{Var}(\saobk{q}_L(\ell_0))}{\operatorname{Var}(\saobk{q}_L(\ell_0))}.
\end{equation}

We want to show the convergence of \eqref{equation_coefficient_error} and \eqref{equation_variance_error} to zero for $\ell_0 \rightarrow +\infty$. We achieve this by using an academic toy model with $L = 4$ and $\ell = 1,\dots,4$, where
\begin{equation} \label{equation_numerics_toy_problem}
\begin{aligned}
Z_\ell(\omega) &= Z(\omega) + \sum_{k = 2}^4 c_k(\omega) 2^{- (k - 1) (\ell + \ell_0)} + 0.1 \xi_\ell(\omega) 2^{- 3 (\ell + \ell_0)}, \\
(Z, c_2, c_3, c_4)^T &\sim N(0, Q), \quad Q_{ij} = \exp(- |i - j|), \quad i,j=1,\dots,4, \\
\xi_\ell &\sim N(0, 1).
\end{aligned}
\end{equation}
We assume that $\xi_i$ and $\xi_j$ are independent for $i \not = j$ and that the $\xi_\ell$ are independent of $Z, c_2, c_3, c_4$. 
The parameter $\ell_0$ controls the accuracy of the coarsest model.
This model satisfies \Cref{assumption_re_expansion} with $q = 4$, $\gamma^2 = 1$, $\gamma^3 = 2$ and $\gamma^4 = 3$. 
Furthermore, we fix artificial costs of $\operatorname{Cost}(Z_\ell) = 4^{\ell - 1}$ and remark that the model covariance matrix $C$ can be computed analytically.

Recall that $\operatorname{RE}, 2$ is the MLMC estimator if the bias $\alpha = e_\ell$. 
We plot the computed values of \eqref{equation_coefficient_error} and \eqref{equation_variance_error} for different values of $\ell_0 = 0,\dots, 6$ in \Cref{figure_coefficients_variance_academic}. 
We conclude that both quantities converge to zero showing that for large $\ell_0$ the estimators $\operatorname{SAOB}, q$ and $\operatorname{RE}, q$ are almost identical for the problem in \eqref{equation_numerics_toy_problem}. 
Notice that we allow fractional samples $m_k \in \mathbb{R}_{\geq 0}$ and do not ceil. 
This does not change the results of this section in a fundamental way, since we could also arbitrarily increase the budget or the variance by scaling. 
\begin{figure} 
	\begin{center}
	\includegraphics[trim=8 0 30 10, clip, width=0.49 \textwidth]{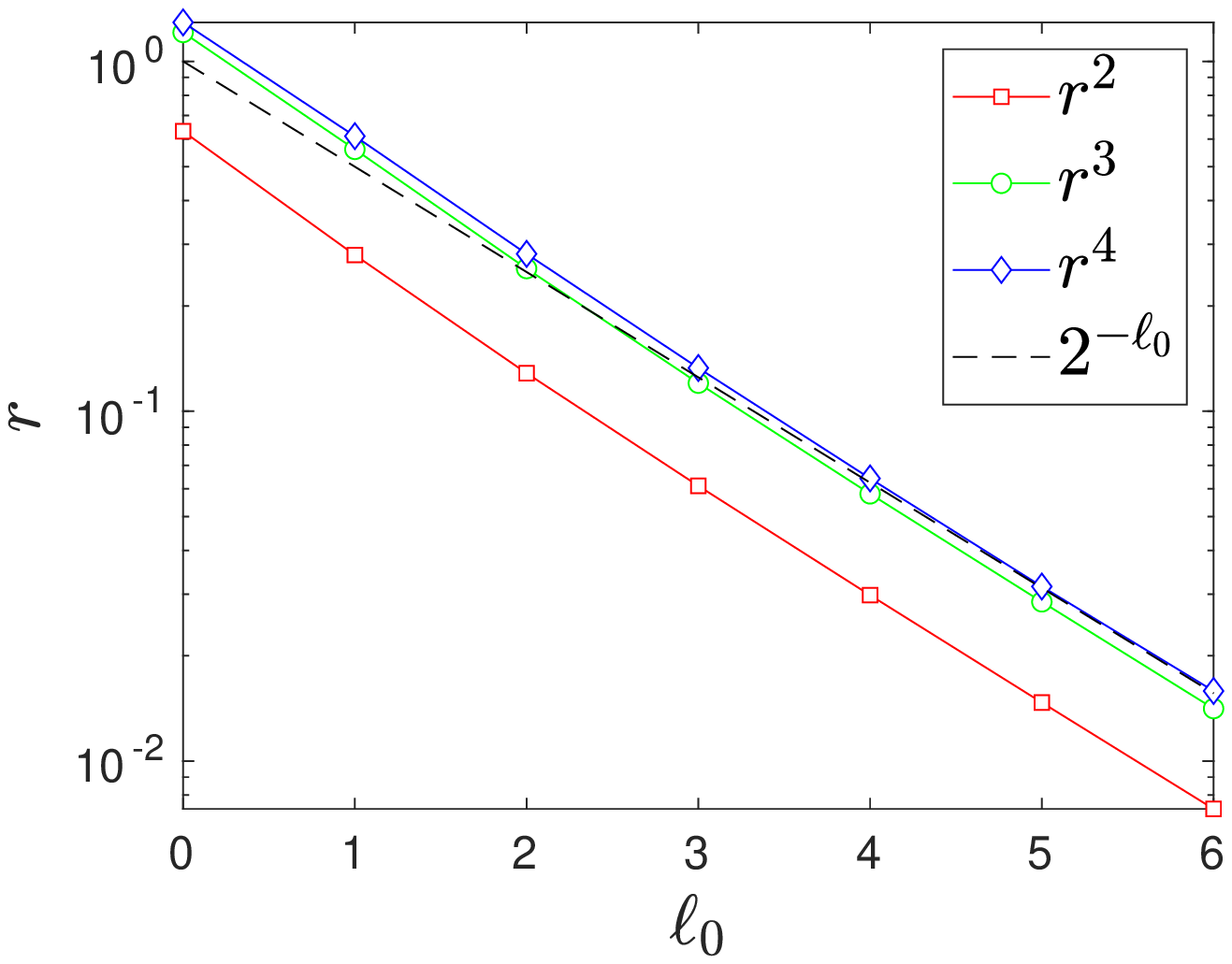}
	\includegraphics[trim=8 0 30 10, clip, width=0.49 \textwidth]{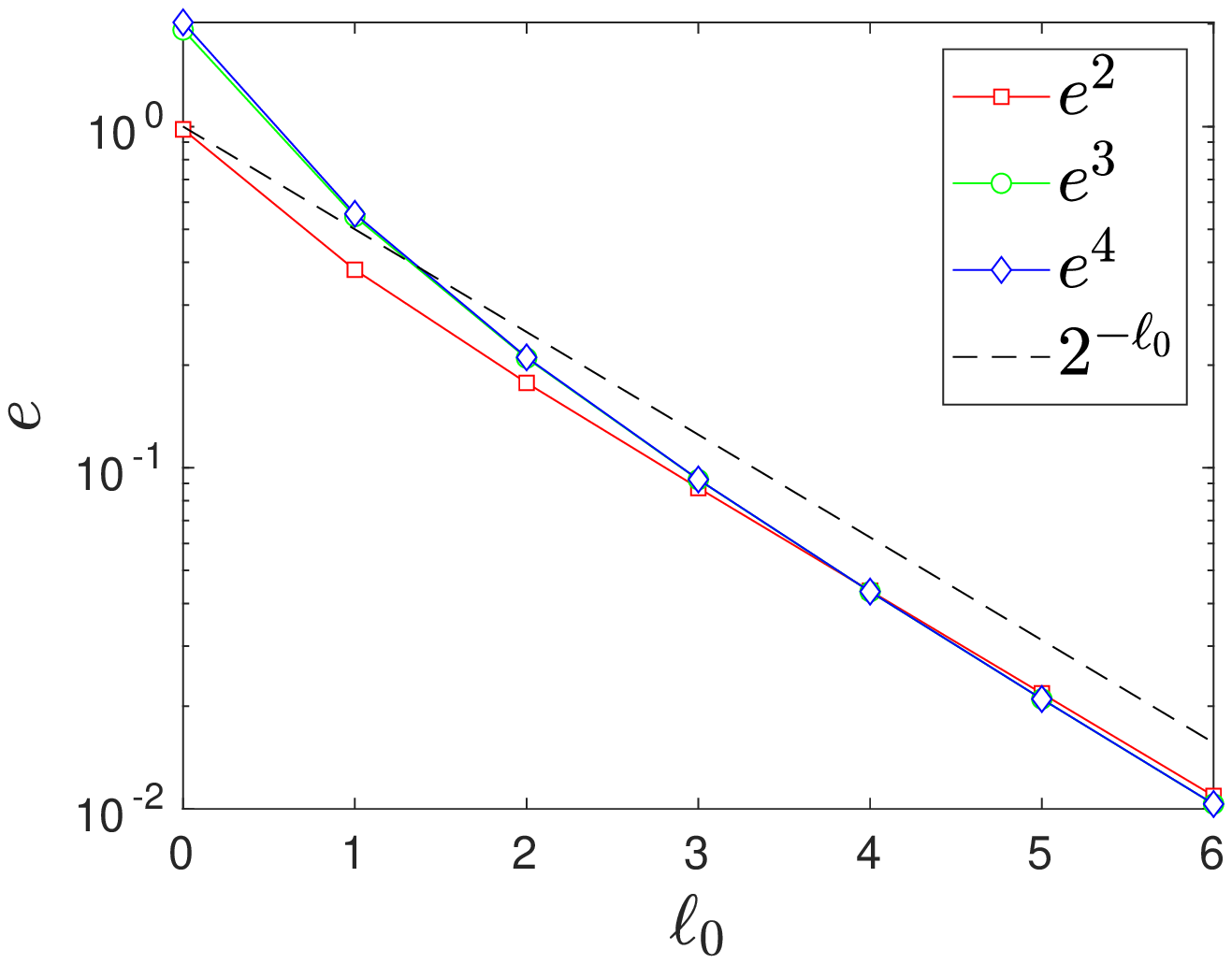}
	\end{center}
	\caption{The left image shows the convergence of the coefficients of the $\operatorname{SAOB}, q$ to the coefficients of the $\operatorname{RE}, q$ estimators w.r.t. $\ell_0$. The right image shows the relative variance to the corresponding $\operatorname{SAOB}, q$ estimator. A reference rate is drawn dashed. Here $r^2$ and $e^2$ is the coefficient and relative variance increment of $\operatorname{RE}, 2$, which is the MLMC estimator.} 
	\label{figure_coefficients_variance_academic}
\end{figure}

\begin{remark}
We informally state an explanation for the convergence of the coefficients of $\operatorname{SAOB}, q$ to $\operatorname{RE}, q$. 
The variance of $\operatorname{RE}, q$ satisfies
\begin{equation*}
\operatorname{Var}(\ree{q}_L) = \sum_{k = 1}^L \frac{\operatorname{Var}\left(\sum_{\ell \in S^k} \beta^k_\ell Z_\ell\right)}{m_k} = \sum_{k = 1}^L \frac{\mathcal{O}(2^{- 2 \gamma^q \ell_0})}{m_k}.
\end{equation*}
To achieve the rate $2^{- 2 \gamma^q \ell_0}$ for every variance term the coefficients $\beta^k$ are often asymptotically uniquely determined. Scaling these coefficients down is not allowed since we have the bias constraint $e_L = \sum_{k = 1}^L \beta^k$. 
The proof of \Cref{lemma_re_bias_variance} shows that the linear combinations $(v^{k, q} - v^{k - 1, q})$ needed to achieve an increased order are uniquely determined with the exception of some degenerate cases, i.e., if the $c_k$ in \Cref{assumption_re_expansion} are zero or linearly dependent. 
\end{remark}
\begin{remark}
The numerical results in this section suggest that the estimators $\operatorname{RE}, q$ converge to $\operatorname{SAOB}, q$ if the coarsest grid is fine enough. We thus suspect that for hierarchical models $\operatorname{RE}, q$ and $\operatorname{SAOB}, q$ often have the same asymptotic cost. We also observed this numerically in \Cref{comp_ell}. $\operatorname{SAOB}$ may couple more models than $\operatorname{RE}, q$, however if the remainder term in \Cref{assumption_re_expansion} cannot be removed by linearly combining more models we expect
\begin{equation*}
(\beta^k)^T C \beta^k \geq c 2^{- 2 k \gamma^q},
\end{equation*}
and thus $\saob_\alpha$ and $\ree{q}_\alpha$ should have the same variance rate of $\gamma^{\operatorname{Var}} = 2 \gamma^q$. 
They also have the same bias rate $\gamma^{\operatorname{Bias}}$. For $\operatorname{SAOB}$ it is reasonable to assume that $S^k = \{1, \dots, k\}$ for $k = 1, \dots, L$, since $\operatorname{SAOB}$ should use all models $Z_1, \dots, Z_L$ if $Z_1$ is fine enough and the asymptotic cost of $S^k$ is given by its finest model, i.e., we may add all coarse models without an asymptotic cost increase. Therefore both estimators have equal bias, variance and cost rates and thus $\saob_\alpha$ has costs asymptotically equal to $\ree{q}_\alpha$ where $q$ is maximal such that \Cref{assumption_re_expansion} is satisfied. 
The reasoning in this remark is however not a formal proof.
\end{remark}
\section{Comparison with ACV estimators} \label{section_acv}
We return to the example in \Cref{sec:numerics} and compare the SAOBs and RE estimators with the ACV estimators \cite{Gorodetsky_2020}, in particular, the ACV-IS, ACV-MF and ACV-KL. 
For brevity we provide only the definition of the ACV-MF estimator,
\begin{equation} \label{equation_acvmf_definition}
\acvmfe_L = \frac{1}{n_L} \sum_{i = 1}^{n_L} Z_L(\omega^i) + \sum_{\ell = 1}^{L - 1} \beta_\ell \left(\frac{1}{n_L} \sum_{i = 1}^{n_L} Z_\ell(\omega^i) - \frac{1}{n_\ell} \sum_{i = 1}^{n_\ell} Z_\ell(\omega^i) \right),
\end{equation} 
and refer to \cite[Def. 2]{Gorodetsky_2020} and \cite[Def. 4]{Gorodetsky_2020} for the definition of the ACV-IS and ACV-KL, respectively.
The comparison is carried out in a separate section since the ACV estimators have been developed very recently, and no asymptotic complexity results are known to date. 
Recall that an important motivation for the introduction of the ACV estimators is their increased variance reduction capacity compared to MLMC and MFMC.
In fact, the ACV estimators in \cite{Gorodetsky_2020} reach the exact same lower variance bound in the infinite low fidelity data limit as the multilevel BLUEs (cf. \cite[Sec. 4]{SchadenUllmann:2020}).
However, the improved variance reduction does not necessarily result in a smaller asymptotic cost for parametric, PDE-based models which we study in this paper. 
We observed this already in the numerical example in \Cref{sec:numerics} for the SAOBs where the asymptotic complexity is improved in some settings, but depends on the bias, variance and cost rate in the considered example.

We plot the cost of the ACV estimators in \Cref{figure_acv_mse} for the artificial cost rate $\gamma^{\operatorname{Cost}} = 6$ without and with ceiling the number of samples. 
\begin{figure} 
	\begin{center}
	\includegraphics[trim=8 0 30 10, clip, width=0.49 \textwidth]{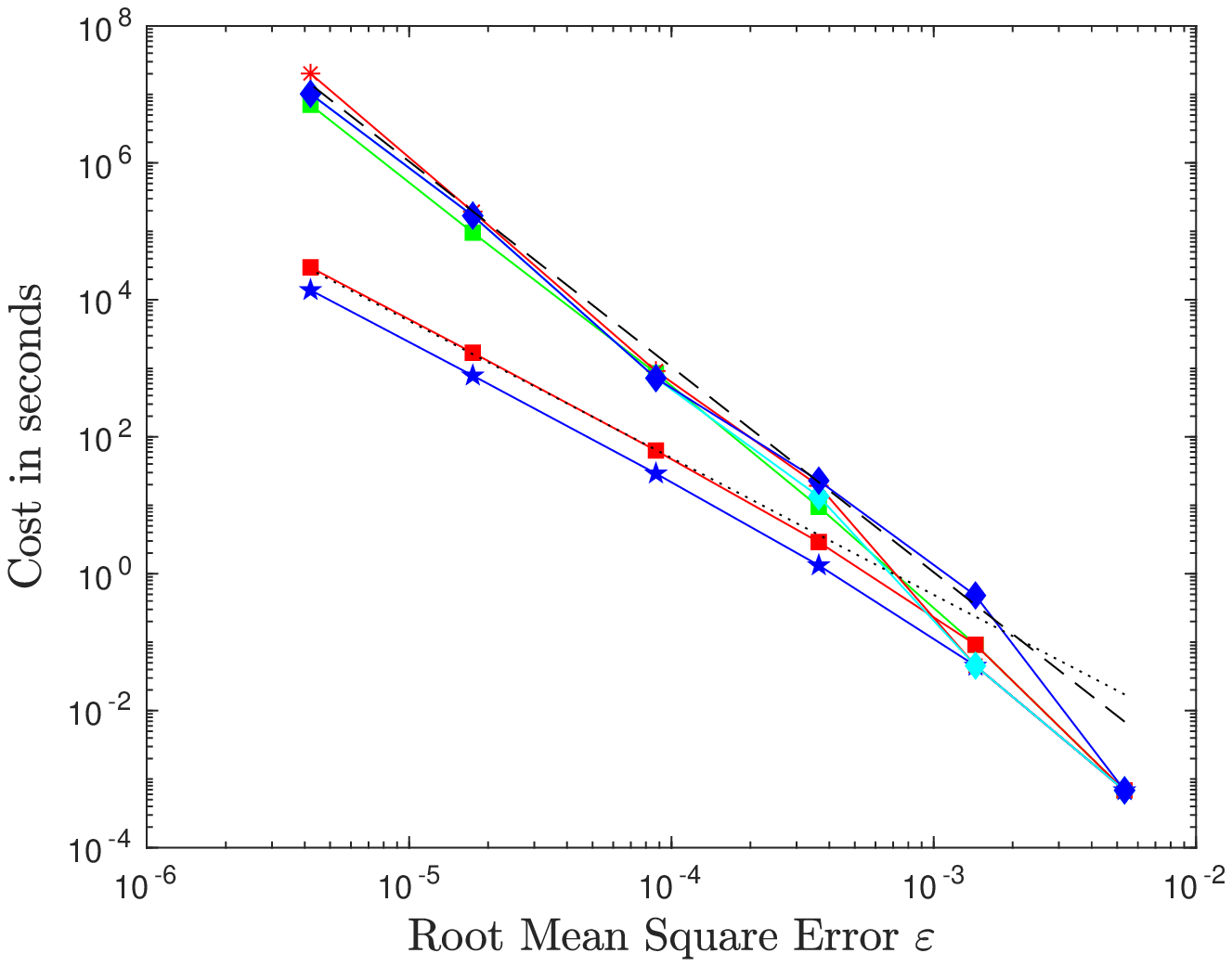}
	\includegraphics[trim=8 0 30 10, clip, width=0.49 \textwidth]{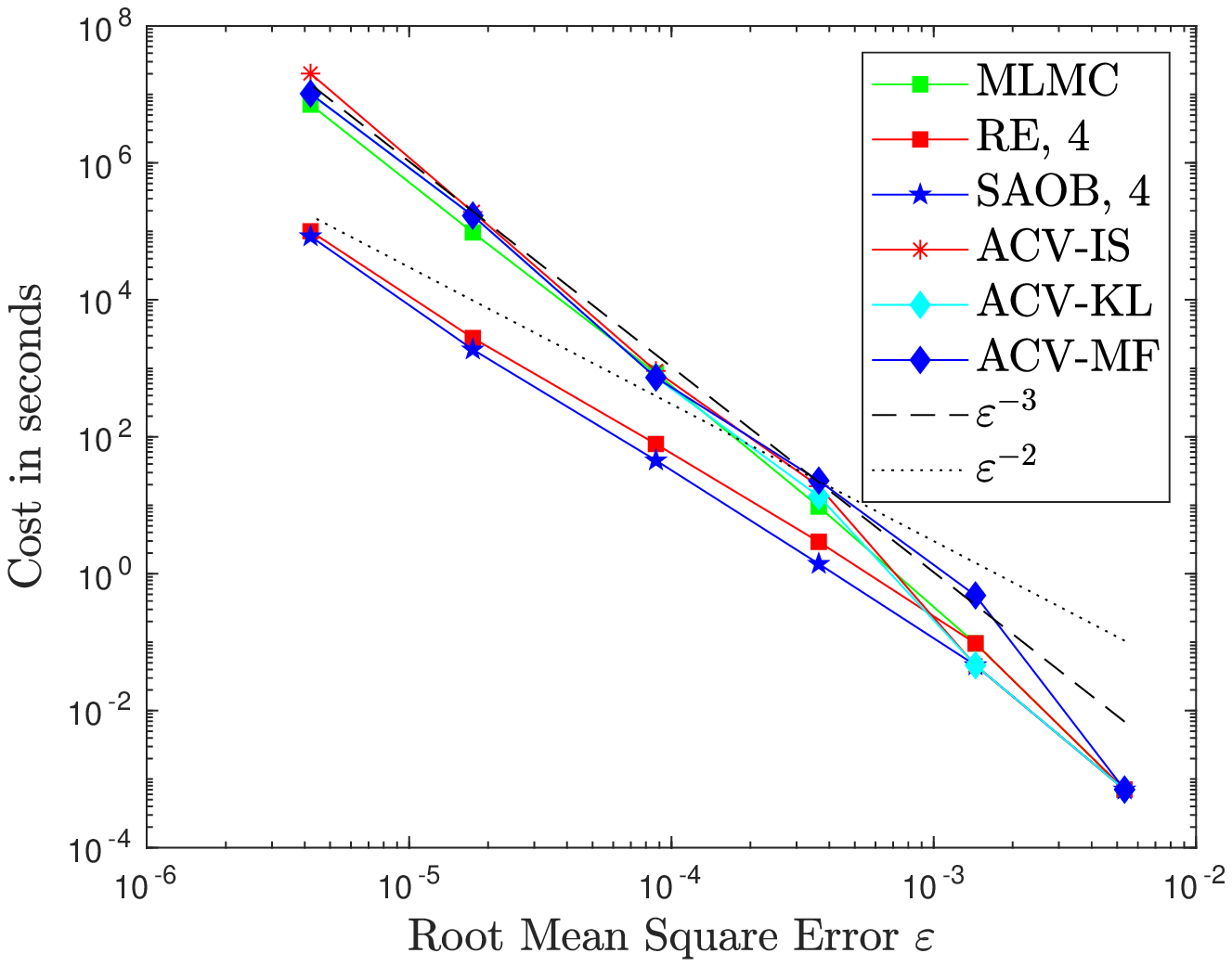}
	\end{center}
	\caption{The left image shows the MSE for the computed costs not ceiling the number of samples and the right image with ceiling for the artificial cost $\gamma^{\operatorname{Cost}} = 6$. For small $\varepsilon$ the cost for the ACV-MF and ACV-KL estimator are identical. Both plots display the estimators with a bias of $\mu_\ell$.} 
	\label{figure_acv_mse}
\end{figure}
We can clearly see that in this example the ACV estimators have the asymptotic cost of order $\varepsilon^{-3}$.
This is the same cost as the other estimators, and in particular, the MLMC estimator.
The results in the right plot of \Cref{figure_acv_mse} should be compared with the left plot in \Cref{figure_complexity_elliptic_artificial}.

We now offer a possible explanation for this observation by looking at the coefficients $\beta^k$ of the various estimators.
Since the ACV estimators are also linear unbiased estimators for $\mu_L$ we compare them to the SAOB, 4 and the weighted RE, 4 estimator. Each of these three estimators can be written as 
\begin{equation*}
\widehat{\mu}_L = \sum_{k = 1}^6 \sum_{\ell \in S^k} \beta^k_\ell \frac{1}{m_k} \sum_{i = 1}^{m_k} Z_\ell(\omega^i_k),
\end{equation*}
with different model groups $S^1, \dots, S^6$ for every estimator. We plot the resulting coefficients $\beta^k$ in \Cref{figure_acv_coefficients}.
\begin{figure} 
	\begin{center}
	\includegraphics[trim=60 10 70 0, clip, width=0.325 \textwidth]{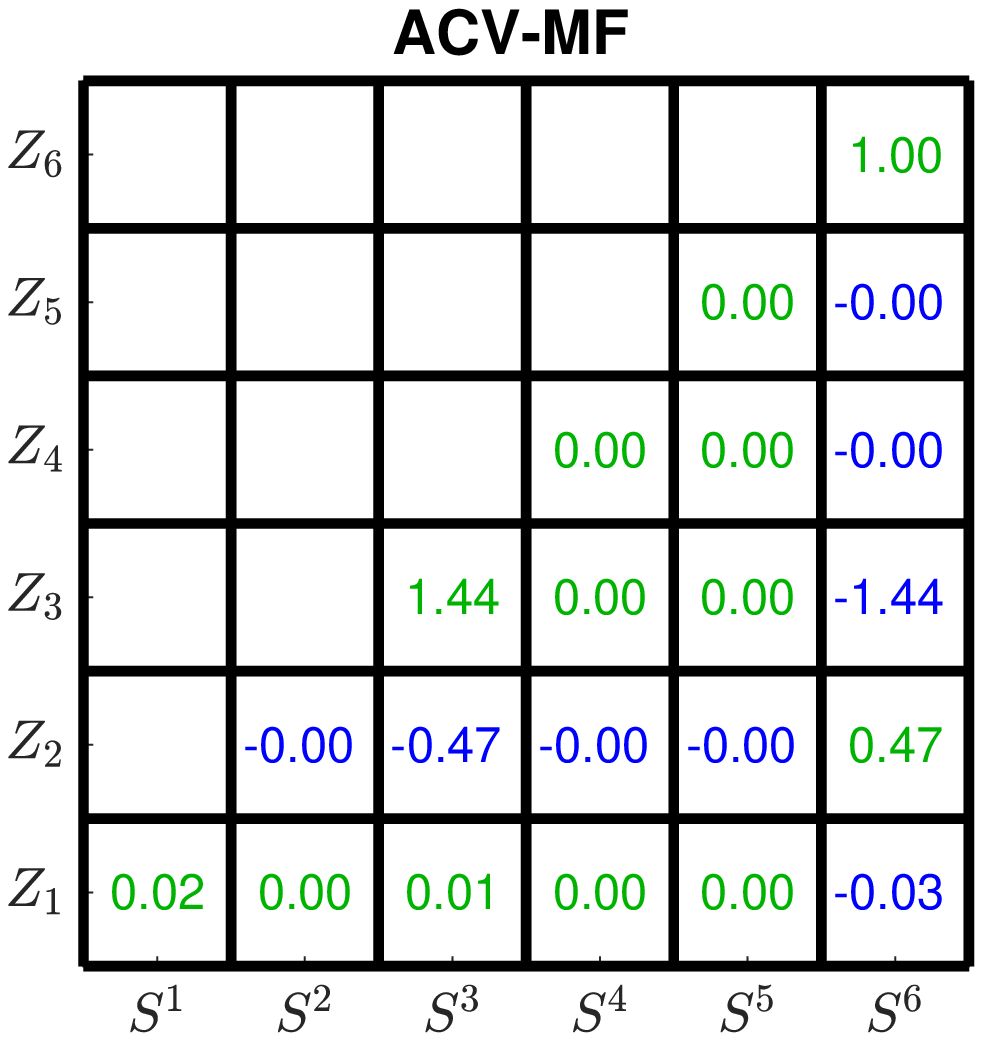}
	\includegraphics[trim=60 10 70 0, clip, width=0.325 \textwidth]{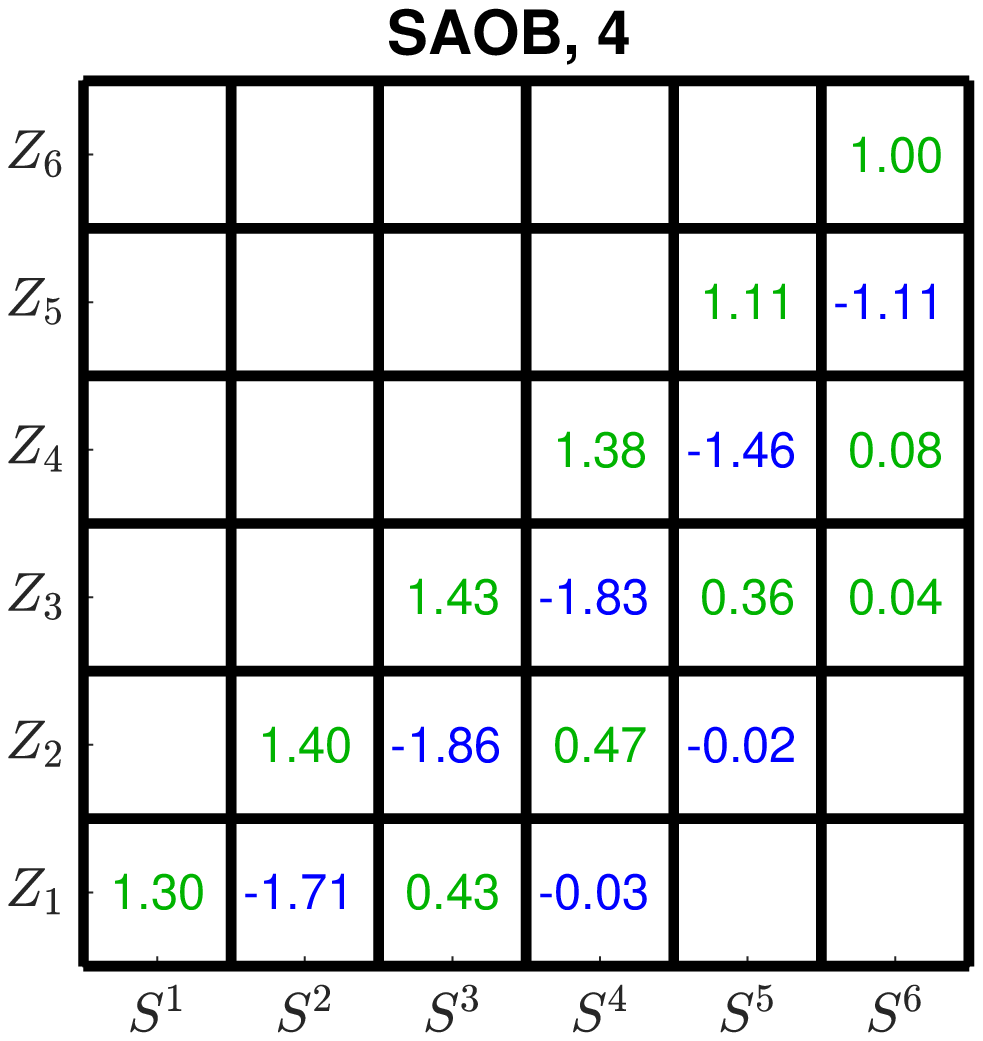}
	\includegraphics[trim=60 10 70 0, clip, width=0.325 \textwidth]{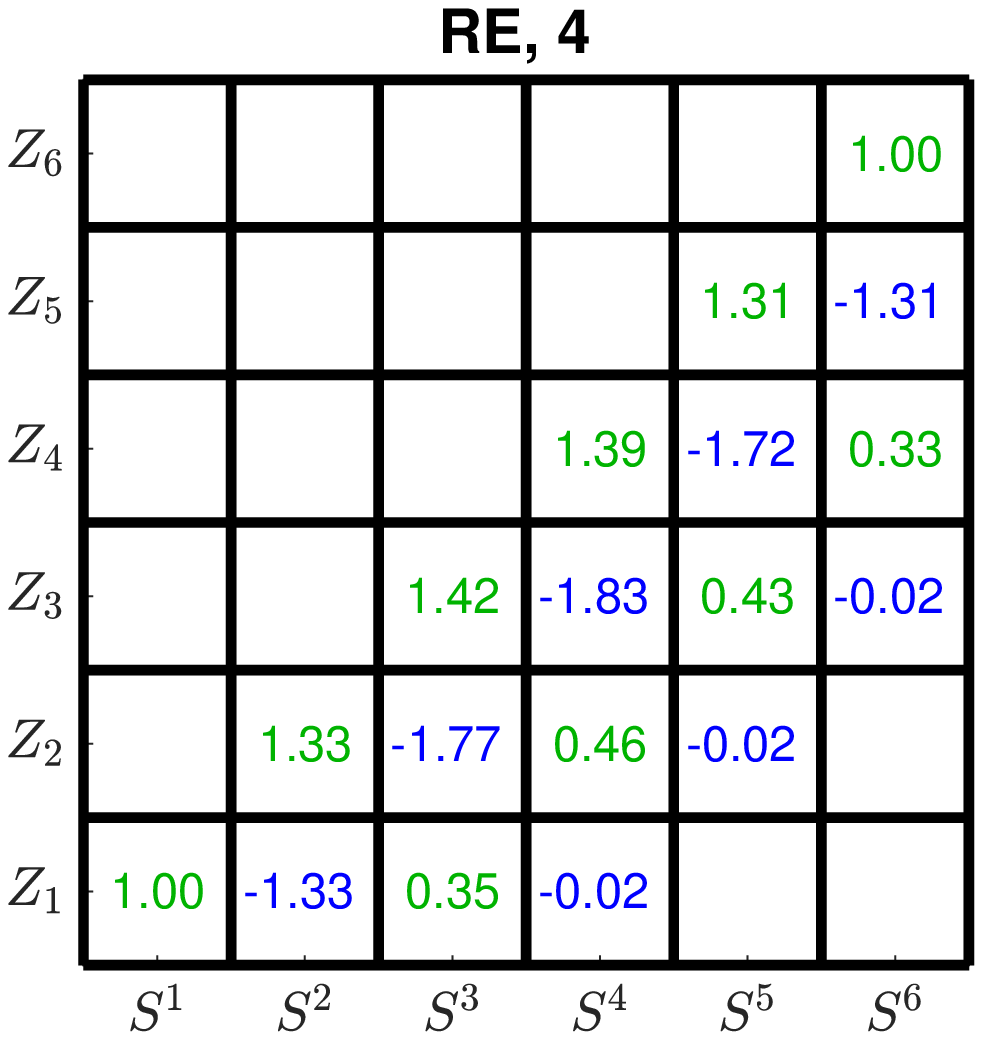}
	\end{center}
	\caption{Coefficients in front of every model group for some estimators in \Cref{figure_acv_mse} where the smallest value of $\varepsilon$ has been used. A column represents the models in the model group. For example, $S^5 = \{2,\dots,5\}$ for RE, 4 and SAOB, 4 and $S^5 = \{1, \dots, 6\}$ for ACV-MF. The coefficient $\beta^4_3 = -1.83$ for both SAOB, 4 and RE, 4. Blank entries are not contained in the respective model group. Entries with values $0.00$ or $-0.00$ are in the respective model group but have a small coefficient. The bias constraint ensures that the column vectors add to $e_L$ (up to small rounding errors due to the display of only two decimal places).} 
	\label{figure_acv_coefficients}
\end{figure}
Using the definition of the ACV-MF estimator in \eqref{equation_acvmf_definition} and using $n_\ell \geq n_L$ it is straightforward to verify that the ACV-MF coefficients satisfy the sign pattern
\begin{equation} \label{equation_acv_sign_equality}
\operatorname{sign}(\beta^L_\ell) = -\operatorname{sign}(\beta^k_\ell) \quad \text{ for all } k \in \{1, \dots, L - 1\} \text{ with } \ell \in S^k.
\end{equation}
However, the RE estimator exhibits a chequerboard pattern w.r.t. the sign of the coefficients, which cannot be satisfied under the conditions \eqref{equation_acv_sign_equality}. 
This reduces the variance reduction capability of the ACV-MF estimator, since some linear combinations of models are prohibited by construction. 
Thus, the estimator variance
\begin{equation*}
\operatorname{Var}\left(\sum_{\ell \in S^k} \beta^k_\ell Z_\ell \right) = (\beta^k)^T C \beta^k
\end{equation*} 
may be larger than necessary. If the linear combinations that asymptotically decrease the estimator variance require multiple sign changes, which seems to be the case according to \Cref{section_re_convergence_to_saob} and \Cref{figure_acv_coefficients}, then the smaller complexity of the RE estimators cannot be achieved by the ACV-MF estimator. Since for the ACV-MF we can have a single sign change similar to the sign pattern in MLMC, we believe that the asymptotic complexity of the ACV-MF estimator is in general equal to the complexity of MLMC for the parametric PDE-based models in this paper. 
We remark that the sign pattern \eqref{equation_acv_sign_equality} holds also true for the ACV-IS and ACV-KL estimator.
Hence, a performance similar to the ACV-MF can be expected.
\section{Conclusions}\label{sec:conclusions}
In this paper we study the asymptotic properties of the multilevel BLUE for the expectation of scalar-valued, PDE-based random outputs. 
The main tool of our analysis is a pathwise expansion of the random output in terms of a discretization parameter, linking a collection of models associated with the output.
We use the idea of Richardson extrapolation (RE) combined with MLMC theory to analyse the complexity of a certain RE estimator.
This allows us to bound the asymptotic complexity of the SAOB which is not worse than the complexity of the RE estimators and in particular MLMC.
Numerical experiments with a smooth PDE-based output in 2D space suggest that a RE type of expansion for the quantity of interest holds both in terms of the bias and variance. 
If we use the true sample cost, then the SAOB, MLMC and MFMC estimators give the optimal complexity with respect to the mean-square error.
For an artificial, increased cost, the SAOB and RE estimators have a smaller complexity compared to MLMC and MFMC. 
However, it remains an open research question whether the SAOB has an asymptotically smaller complexity than the RE estimator.  
\section*{Acknowledgements}
The authors thank Michael Ulbrich for the finite element code that was used to solve the PDE and to sample from the mean zero Gaussian random field with Matern 3/2 covariance in \Cref{sec:numerics}. 
\appendix
\section{Proof of Lemma~\ref{lemma_opt_sample_allocation}}
\label{appendix_proof_lemma_opt_sample_allocation}
%
%
First, we remark that \eqref{equation_saob_opt_samples} is valid if $\beta^k = 0$, since then by our convention $m^*_k = 0$. On the other hand, if $\beta^k \not = 0$ and since $C$ is positive definite, we have
\begin{equation*}
\lim_{m_k \rightarrow 0} J(m) = +\infty.
\end{equation*}
Furthermore, since \eqref{equation_opt_sample_allocation} clearly has a feasible point, the number of samples is lower bounded,
\begin{equation*}
m^*_k > 0 \quad \text{for all } k \in \{1,..., K\} \text{ with } \beta^k \not = 0.
\end{equation*}
Similarly, since $W_k > 0$, the cost constraint ensures the upper bound
\begin{equation*}
m^*_k < c \quad \text{for all } k \in \{1,..., K\}.
\end{equation*}
Hence \eqref{equation_opt_sample_allocation} is an optimization problem over a compact set and since $J$ is convex and thus continuous, a minimizer $m^*$ exists. The function $J$ is monotonically decreasing in $m$ and thus the cost constraint is satisfied with equality at a minimizer 
\begin{equation*}
\sum_{k = 1}^K m^*_k W_k = p,
\end{equation*}
which allows us to replace the inequality constraint with an equality constraint. 
The Karush--Kuhn--Tucker conditions \cite[Section 5.5]{boyd_2009} are necessary and sufficient for the convex optimization problem \eqref{equation_opt_sample_allocation}. 
Thus, with Lagrange-Multipliers $\lambda \in \mathbb{R}$ and $\xi_1, \dots, \xi_K \in \mathbb{R}$,
\begin{align}
\label{equation_KKT_1}  - \frac{(\beta^k)^T C \beta^k}{(m^*_k)^2} + \lambda W_k - \xi_k  &= 0 \qquad \text{for all } k = \{1, \dots, K\} \text{ with } \beta^k \not = 0, \\
\label{equation_KKT_2} \sum_{k = 1}^K m^*_k W_k &= p, \\
m^*_k \geq 0, \quad \xi_k &\geq 0, \quad \xi_k m^*_k = 0 \qquad \text{for all } k = \{1, \dots, K\} \text{ with } \beta^k \not = 0
\end{align} 
and $m^*_k = 0$ if $\beta^k = 0$. For $\beta^k \not = 0$ we have already shown $m^*_k > 0$, thus $\xi_k = 0$ and \eqref{equation_KKT_1} reads
\begin{equation} \label{equation_mstar_temp}
m^*_k = \left(\frac{(\beta^k)^T C \beta^k}{\lambda W_k} \right)^{1 / 2}.
\end{equation}
This expression is well defined for $\lambda > 0$, which is valid since otherwise \eqref{equation_KKT_1} cannot be satisfied at the minimizer $m^*$. 
We insert \eqref{equation_mstar_temp} into the cost constraint \eqref{equation_KKT_2} and arrive at
\begin{equation} \label{equation_lambdastar_temp}
\frac{1}{\lambda^{1 / 2}} = \frac{p}{\sum_{k = 1}^K ((\beta^k)^T C \beta^k W_k )^{1 / 2}}.
\end{equation}
Here the denominator is positive since $W_k > 0$, $C$ is positive definite and $\beta^k \not = 0$ for some index $k$. 
The last statement follows from the bias constraint $\alpha = \sum_{k = 1}^K \beta^k$ and $\alpha \not = 0$ from the assumptions of this lemma. 
We now insert \eqref{equation_lambdastar_temp} into \eqref{equation_mstar_temp} to obtain the result \eqref{equation_saob_opt_samples}, which also shows the uniqueness of a minimizer. 
Inserting this $m^*$ into $J$ then shows \eqref{equation_saob_opt_variance}.
%
%
\section{Proof of Lemma~\ref{lemma_re_bias_variance}}
\label{appendix_proof_lemma_re_bias_variance}
%
%
We first show that the following statement is true for all $k = 1,\dots, L$ and $\ell_0 \in \{0, \dots, L - \ell\}$:
\begin{equation} \label{equation_re_small_expansion}
\sum_{\ell = 1}^{L - \ell_0} v^{k, q}_\ell Z_{\ell + \ell_0} = Z + \sum_{j = k + 1}^{q - 1} c^k_j 2^{- (k + \ell_0) \gamma^j} + \mathcal{O}(2^{- (k + \ell_0) \gamma^q}),
\end{equation}
where $c^k_j$ are random variables with bounded second moment. 
The remainder term $\mathcal{O}(\cdot)$ also has a bounded second moment. 
The statement for $k = 1$ is \Cref{assumption_re_expansion} (i). 
Now let $1 < k < q$ and assume the induction hypothesis is true for $k - 1$. 
Observe that $v^{k - 1}_{L - \ell_0} = 0$ if $\ell_0 \in \{0, \dots, L - k\}$ due to the recursion \eqref{equation_re_vectors}.
Defining $v^{k - 1, q}_0 = 0$ we have
\begin{align*}
\sum_{\ell = 1}^{L - \ell_0} v^{k, q}_\ell Z_{\ell + \ell_0} &= \frac{1}{2^{\gamma^k} - 1} \sum_{\ell = 1}^{L - \ell_0} (2^{\gamma^k} v^{k - 1, q}_{\ell - 1} - v^{k - 1, q}_\ell) Z_{\ell + \ell_0} \\
&= \frac{1}{2^{\gamma^k} - 1} \left(2^{\gamma^k} \sum_{\ell = 1}^{L - \ell_0} v^{k - 1, q}_\ell Z_{\ell + \ell_0 + 1} - \sum_{\ell = 1}^{L - \ell_0} v^{k - 1, q}_\ell Z_{\ell + \ell_0} \right).
\end{align*}
Now we apply the induction hypothesis, which leads to
\begin{align*}
\sum_{\ell = 1}^{L - \ell_0} v^{k, q}_\ell Z_{\ell + \ell_0} &= Z + \frac{2^{\gamma^k}}{2^{\gamma^k} - 1} \left(\sum_{j = k}^{q - 1} c^{k - 1}_j 2^{- (k + \ell_0) \gamma^j} + \mathcal{O}(2^{- (k + \ell_0) \gamma^q}) \right) \\
&- \frac{1}{2^{\gamma^k} - 1} \left(\sum_{j = k}^{q - 1} c^{k - 1}_j 2^{- (k + \ell_0 - 1) \gamma^j} + \mathcal{O}(2^{- (k + \ell_0 - 1) \gamma^q}) \right).
\end{align*}
The random coefficients thus satisfy
\begin{equation*}
c^k_j = \frac{1}{2^{\gamma^k} - 1} \left(2^{\gamma^k} c^{k - 1}_j - 2^{\gamma^j} c^{k - 1}_j \right),
\end{equation*}
where a similar expression for the remainder term is valid. 
Notice that $c^k_k = 0$ and that the $c^k_j$ have bounded second moment as well as the remainder term. 
Therefore \eqref{equation_re_small_expansion} holds for $1 \leq k < q$. 
For $k \geq q$ observe $v^{k - 1, q}_{L - \ell_0} = 0$ and thus
\begin{equation*}
\sum_{\ell = 1}^{L - \ell_0} v^{k, q}_\ell Z_{\ell + \ell_0} = \sum_{\ell = 1}^{L - \ell_0} v^{k - 1, q}_{\ell - 1} Z_{\ell + \ell_0} = \sum_{\ell = 1}^{L - \ell_0 - 1} v^{k - 1, q}_\ell Z_{\ell + \ell_0 + 1} = \sum_{\ell = 1}^{L - \ell_0} v^{k - 1, q}_\ell Z_{\ell + \ell_0 + 1},
\end{equation*}
which by repeatedly applying this process allows us to reduce the case of $k \geq q$ to $k = q - 1$ by increasing the value of $\ell_0$. This shows \eqref{equation_re_small_expansion}.

Let us now prove the bias estimate \eqref{equation_M1_re} using \eqref{equation_re_small_expansion}. For $\ell_0 = 0$ we conclude
\begin{equation*}
|(v^{k, q})^T \mu - \mathbb{E}[Z]| = \left| \sum_{j = k + 1}^{q - 1} \mathbb{E}[c^k_j] 2^{- k \gamma^j} + \mathcal{O}(2^{- k \gamma^q}) \right| \leq c \begin{cases} 
2^{- k \gamma^{k + 1}}, &\text{if } k < q, \\
2^{- k \gamma^q}, &\text{otherwise}.
\end{cases}
\end{equation*}
The bound \eqref{equation_M1_re} for the case $k < q$ is obtained using the crude estimate
\begin{equation*}
2^{- k \gamma^{k + 1}} = 2^{- k \gamma^{k + 1}} 2^{k \gamma^q} 2^{- k \gamma^q} \leq 1 \cdot 2^{q \gamma^q} 2^{- k \gamma^q} \leq c 2^{- k \gamma^q},
\end{equation*}
where now $c$ is independent of $k$. The variance estimate \eqref{equation_M2_re} can be derived similarly from \eqref{equation_re_small_expansion}. Here the key idea is that the difference $v^{k, q} - v^{k - 1, q}$ is used to remove $Z$ from \eqref{equation_re_small_expansion}, which we use with $\ell_0 = 0$, $k$ and $k - 1$,
\begin{align*}
\sum_{\ell = 1}^L (v^{k, q}_\ell - v^{k - 1, q}_\ell) Z_\ell &= \sum_{j = k + 1}^{q - 1} c^k_j 2^{- k \gamma^j} + \mathcal{O}(2^{- k \gamma^q}) - \sum_{j = k}^{q - 1} c^{k - 1}_j 2^{- (k - 1) \gamma^j} + \mathcal{O}(2^{- (k - 1) \gamma^q}) \\
&= \sum_{j = k}^{q - 1} \widetilde{c}^{k - 1}_j 2^{- (k - 1) \gamma^j} + \mathcal{O}(2^{- (k - 1) \gamma^q})
\end{align*}
for suitably defined random variables $\widetilde{c}^{k - 1}_j$ with bounded second moment. Thus taking the variance yields the desired result. Finally, the estimate \eqref{equation_M3_re} follows from the definition of the model groups in \eqref{equation_re_model_groups} and the geometric cost of $Z_\ell$ \eqref{equation_re_geometric_cost} in \Cref{assumption_re_expansion} (ii),
\begin{align*}
W_k &= \operatorname{Cost}\left(S^k \right) = \operatorname{Cost}\left(Z_{\max\{k - q + 1, 1\}}, Z_{\max\{k - q + 2, 1\}}, \dots, Z_{k} \right) = \sum_{\ell = \max \{k - q + 1, 1\}}^k \operatorname{Cost}(Z_\ell) \\
&\leq c \sum_{\ell = 1}^k 2^{\ell \gamma^{\operatorname{Cost}}} \leq c 2^{k \gamma^{\operatorname{Cost}}}.
\end{align*}
%
%
\section{Proof of Lemma~\ref{lemma_weighted_re_bias_variance}}
\label{appendix_proof_lemma_weighted_re_bias_variance}
%
%
We only have to prove the boundedness of $a_k$. For $t = s$ we have $a_k = 1$ and the standard RE estimator. For $s < t$ observe from \eqref{equation_re_vectors} that there exist coefficients $d$ such that 
\begin{equation*}
\sum_{k = 1}^t d_k v^{k, s} = v^{t, t}.
\end{equation*} 
We use the property of the shift matrix $D$ to obtain  
\begin{equation*}
\sum_{k = L - t + 1}^{L} d_{k - L + t} v^{k, s} = \sum_{k = 1}^t d_k v^{L - t + k, s} = D^{L - t} v^{t, t} = v^{L, t}.
\end{equation*}
We combine this with the basis property of the differences $v^{k, s} - v^{k - 1, s}$ in \eqref{equation_re_basis}, we define $a_{L + 1} = 0$ and use $v^{0, s} = 0$ to rewrite the sum in terms of differences of the $a_k$,
\begin{equation*}
v^{L, t} = \sum_{k = 1}^L a_k (v^{k, s} - v^{k - 1, s}) = \sum_{k = 1}^L a_k (v^{k, s} - v^{k - 1, s}) = \sum_{k = 1}^L (a_k - a_{k + 1}) v^{k, s}.
\end{equation*}
We summarize the chain of equations 
\begin{equation*}
\sum_{k = L - t + 1}^{L} d_{k - L + t} v^{k, s} = \sum_{k = 1}^L (a_k - a_{k + 1}) v^{k, s}
\end{equation*}
and since the $v^{k, s}$ are linearly independent the last $t$ coefficients $a_{L - t + 1}, \dots, a_L$ only depend on $d$ and are independent of $L$. The remaining coefficients satisfy $a_k = a_{k + 1}$ for $\ell = 1, \dots, L - t$ and thus $|a_k| \leq c$ with a constant $c$ independent of $k$ and the finest level $L$.
%
%
\bibliographystyle{siamplain}
\bibliography{literature}

\begin{thebibliography}{10}

\bibitem{Asadzadeh_2009}
{\sc M.~Asadzadeh, A.~H. Schatz, and W.~Wendland}, {\em {A new approach to
  Richardson extrapolation in the finite element method for second order
  elliptic problems}}, Mathematics of Computation, 78 (2009), pp.~1951--1973,
  \url{https://doi.org/10.1090/S0025-5718-09-02241-8}.

\bibitem{Blum_1986}
{\sc H.~Blum, Q.~Lin, and R.~Rannacher}, {\em Asymptotic error expansion and
  {R}ichardson extrapolation for linear finite elements}, Numer. Math., 49
  (1986), pp.~11--37, \url{https://doi.org/10.1007/BF01389427}.

\bibitem{boyd_2009}
{\sc S.~Boyd and L.~Vandenberghe}, {\em Convex Optimization}, Cambridge
  University Press, Cambridge, UK, 2004,
  \url{https://doi.org/10.1017/CBO9780511804441}.

\bibitem{Brezinski_2000}
{\sc C.~Brezinski}, {\em Convergence acceleration during the 20th century},
  vol.~122, 2000, pp.~1--21,
  \url{https://doi.org/10.1016/S0377-0427(00)00360-5}.
\newblock Numerical analysis 2000, Vol. II: Interpolation and extrapolation.

\bibitem{Bulirsch_1966}
{\sc R.~Bulirsch and J.~Stoer}, {\em Numerical treatment of ordinary
  differential equations by extrapolation methods}, Numerische Mathematik, 8
  (1966), pp.~1--13, \url{https://doi.org/10.1007/BF02165234}.

\bibitem{Cliffe_2011}
{\sc K.~A. Cliffe, M.~B. Giles, R.~Scheichl, and A.~L. Teckentrup}, {\em
  Multilevel {M}onte {C}arlo methods and applications to elliptic {PDE}s with
  random coefficients}, Comput. Vis. Sci., 14 (2011), pp.~3--15,
  \url{https://doi.org/10.1007/s00791-011-0160-x}.

\bibitem{giles_2008}
{\sc M.~B. Giles}, {\em {Multi-level Monte Carlo path simulation}}, {Operations
  Research}, 56 (2008), pp.~607--617,
  \url{https://doi.org/10.1287/opre.1070.0496}.

\bibitem{giles_2015}
{\sc M.~B. Giles}, {\em {Multilevel Monte Carlo methods}}, Acta Numerica, 24
  (2015), pp.~259--328, \url{https://doi.org/10.1017/S096249291500001X}.

\bibitem{Gorodetsky_2020}
{\sc A.~A. Gorodetsky, G.~Geraci, M.~S. Eldred, and J.~D. Jakeman}, {\em A
  generalized approximate control variate framework for multifidelity
  uncertainty quantification}, Journal of Computational Physics, 408 (2020),
  p.~109257, \url{https://doi.org/10.1016/j.jcp.2020.109257}.

\bibitem{lemaire_2017}
{\sc V.~Lemaire and G.~Pag{\`e}s}, {\em {Multilevel Richardson-Romberg
  extrapolation}}, {Bernoulli}, 20 (2017), pp.~1029--1067,
  \url{https://doi.org/10.3150/16-BEJ822}.

\bibitem{Mbaye2017}
{\sc C.~Mbaye, G.~Pag\`es, and F.~Vrins}, {\em An antithetic approach of
  multilevel {R}ichardson-{R}omberg extrapolation estimator for
  multidimensional {SDE}s}, in Numerical analysis and its applications,
  vol.~10187 of Lecture Notes in Comput. Sci., Springer, Cham, 2017,
  pp.~482--491.

\bibitem{Mueller2015}
{\sc E.~H. M\"{u}ller, R.~Scheichl, and T.~Shardlow}, {\em Improving multilevel
  {M}onte {C}arlo for stochastic differential equations with application to the
  {L}angevin equation}, Proc. A., 471 (2015), pp.~20140679, 20,
  \url{https://doi.org/10.1098/rspa.2014.0679}.

\bibitem{Peherstorfer_2018}
{\sc B.~Peherstorfer, M.~Gunzburger, and K.~Willcox}, {\em {Convergence
  analysis of multifidelity Monte Carlo estimation}}, Numerische Mathematik,
  139 (2018), pp.~683--707, \url{https://doi.org/10.1007/s00211-018-0945-7}.

\bibitem{Peherstorfer_2016}
{\sc B.~Peherstorfer, K.~Willcox, and M.~Gunzburger}, {\em {Optimal Model
  Management for Multifidelity Monte Carlo Estimation}}, SIAM Journal on
  Scientific Computing, 38 (2016), pp.~A3163--A3194,
  \url{https://doi.org/10.1137/15M1046472}.

\bibitem{Peherstorfer_2018b}
{\sc B.~Peherstorfer, K.~Willcox, and M.~Gunzburger}, {\em Survey of
  multifidelity methods in uncertainty propagation, inference, and
  optimization}, SIAM Rev., 60 (2018), pp.~550--591,
  \url{https://doi.org/10.1137/16M1082469}.

\bibitem{Rannacher_1988}
{\sc R.~Rannacher}, {\em Extrapolation techniques in the finite element method
  (a survey)}, no.~MATC7 in Proc. Summer School on Numerical Analysis,
  Helsinki, Univ. of Tech., 1988, pp.~80--113.

\bibitem{Richardson_1911}
{\sc L.~F. Richardson}, {\em The approximate arithmetical solution by finite
  differences of physical problems involving differential equations, with an
  application to the stresses in a masonry dam}, Philosophical Transactions of
  the Royal Society of London. Series A, 210 (1911), pp.~307--357,
  \url{https://doi.org/10.1098/rsta.1911.0009}.

\bibitem{Romberg_1955}
{\sc W.~Romberg}, {\em {Vereinfachte numerische Integration}}, Det Kongelige
  Norske Videnskabers Selskab Forhandlinger, 28 (1955), pp.~30--36.

\bibitem{SchadenUllmann:2020}
{\sc D.~Schaden and E.~Ullmann}, {\em On multilevel best linear unbiased
  estimators}, SIAM/ASA Journal on Uncertainty Quantification, 8 (2020),
  pp.~601--635, \url{https://doi.org/10.1137/19M1263534}.

\bibitem{Stein_1999}
{\sc M.~L. Stein}, {\em Interpolation of spatial data}, Springer Series in
  Statistics, Springer-Verlag, New York, 1999,
  \url{https://doi.org/10.1007/978-1-4612-1494-6}.
\newblock Some theory for Kriging.

\bibitem{TalayTubaro1990}
{\sc D.~Talay and L.~Tubaro}, {\em Expansion of the global error for numerical
  schemes solving stochastic differential equations}, Stochastic Anal. Appl., 8
  (1990), pp.~483--509 (1991), \url{https://doi.org/10.1080/07362999008809220}.

\end{thebibliography}
\end{document}